\documentclass[11pt]{article}
\usepackage{amsmath, amssymb, amsthm}
\usepackage{verbatim}
\usepackage{multicol}
\usepackage{enumerate}
\usepackage{comment}
\usepackage[none]{hyphenat}
\usepackage[unicode]{hyperref}
\hypersetup{
	colorlinks=true,
	linkcolor=blue,
	filecolor=magenta,      
	urlcolor=cyan,
	citecolor=blue
}
\usepackage{pgf}
\usepackage{tikz}
\usetikzlibrary{positioning,arrows,shapes,decorations.markings,decorations.pathreplacing,matrix,patterns}
\tikzstyle{vertex}=[circle,draw=black,fill=black,inner sep=0,minimum size=3pt,text=white,font=\footnotesize]
\usepackage{cleveref}

\date{}
\title{\vspace{-1.2cm} Ramsey properties of algebraic graphs and hypergraphs}

\author{
Benny Sudakov\thanks{ETH Zurich, \emph{e-mail}: \textbf{\{benjamin.sudakov,istvan.tomon\}@math.ethz.ch}}
\and
Istv\'an Tomon\footnotemark[1]
}

\theoremstyle{plain}
\newtheorem{theorem}{Theorem}[section]
\newtheorem{definition}[theorem]{Definition}
\newtheorem{corollary}[theorem]{Corollary}
\newtheorem{claim}[theorem]{Claim}
\newtheorem{lemma}[theorem]{Lemma}

\Crefname{theorem}{Theorem}{Theorems}
\Crefname{definition}{Definition}{Definitions}
\Crefname{corollary}{Corollary}{Corollaries}
\Crefname{claim}{Claim}{Claims}
\Crefname{lemma}{Lemma}{Lemmas}
\Crefname{conjecture}{Conjecture}{Conjectures}
\Crefname{problem}{Problem}{Problems}
\Crefname{prop}{Proposition}{Propositions}

\theoremstyle{definition}

\DeclareMathOperator{\frank}{frank}
\DeclareMathOperator{\mfrank}{mfrank}

\newcommand{\m}{\mathbf}

\begin{document}

\maketitle
\sloppy

\begin{abstract}
    One of the central questions in Ramsey theory asks 
    how small can be the size of the largest clique and independent set in a graph on $N$ vertices. By the celebrated result of Erd\H{o}s from 1947, the random graph on $N$ vertices
    with edge probability $1/2$, contains no clique or independent set larger than $2\log_2 N$, with high probability. Finding explicit constructions of graphs with similar Ramsey-type properties is a famous open problem. A natural approach is to construct such graphs using algebraic tools.
    
    Say that an $r$-uniform hypergraph $\mathcal{H}$ is \emph{algebraic of complexity $(n,d,m)$} if the vertices of $\mathcal{H}$ are elements of $\mathbb{F}^{n}$ for some field $\mathbb{F}$, and there exist $m$ polynomials $f_1,\dots,f_m:(\mathbb{F}^{n})^{r}\rightarrow \mathbb{F}$ of degree at most $d$ such that the edges of $\mathcal{H}$ are  determined by the zero-patterns of $f_1,\dots,f_m$. The aim of this paper is to show that if an algebraic graph (or hypergraph) of complexity $(n,d,m)$ has good Ramsey properties, then at least one of the parameters $n,d,m$ must be large.
    
    In 2001, R\'onyai, Babai and Ganapathy considered the bipartite variant of the Ramsey problem and proved that if $G$ is an algebraic graph of complexity $(n,d,m)$ on $N$ vertices, then either $G$ or its complement contains a complete balanced bipartite graph of size $\Omega_{n,d,m}(N^{1/(n+1)})$. We extend this result by showing that such $G$ contains either a clique or an independent set of size $N^{\Omega(1/ndm)}$ and prove similar results for algebraic hypergraphs of constant complexity. 
    We also obtain a polynomial regularity lemma for $r$-uniform algebraic hypergraphs that are defined by a single polynomial, that might be of independent interest.
    Our proofs combine algebraic, geometric and combinatorial tools. 
\end{abstract}

\section{Introduction}

The quantitative version of Ramsey's theorem, proved by Erd\H{o}s and Szekeres \cite{ESz35}, tells us that every graph on $N$ vertices contains a clique or an independent set of size at least $\frac{1}{2}\log_2 N$. In 1947, Erd\H{o}s \cite{E47} proved that this bound is best possible up to the constant factor, as the random graph on $N$ vertices with edge probability $1/2$ contains no clique or independent set of size larger than $2\log_2 N$, with high probability. Since then, it became a central  problem in graph theory to find explicit constructions of graphs having only logarithmic sized cliques and independent sets. One natural approach to construct such graphs is to use algebraic tools.

Let $r,n,d,m$ be positive integers. Say that an $r$-uniform hypergraph $\mathcal{H}$ is \emph{algebraic of complexity $(n,d,m)$} if the following holds. The vertex set of $\mathcal{H}$ is a subset of $\mathbb{F}^{n}$, where $\mathbb{F}$ is some field, and there exist $m$ polynomials $f_1,\ldots,f_m:(\mathbb{F}^{n})^{r}\rightarrow \mathbb{F}$ of degree at most $d$ and a Boolean formula $\phi:\{\mbox{false},\mbox{true}\}^{m}\rightarrow \{\mbox{false},\mbox{true}\}$ such that $\{\mathbf{v}_1,\dots,\mathbf{v}_r\}\in V(\mathcal{H})^{(r)}$ is an edge if and only if 
\begin{equation}\label{equ:def}
    \phi([f_{1}(\mathbf{v}_1,\dots,\mathbf{v}_r)=0],\dots,[f_{m}(\mathbf{v}_1,\dots,\mathbf{v}_r)=0])=\mbox{true},
\end{equation} i.e., the edges of $\mathcal{H}$ are determined by zero patterns
of polynomials $f_1,\ldots,f_m$.
We assume that for $\{\mathbf{v}_1,\dots,\mathbf{v}_r\}\in V(\mathcal{H})^{(r)}$, the left hand side of (\ref{equ:def}) is invariant under the permutation of $\{\mathbf{v}_1,\dots,\mathbf{v}_r\}$, so the edges of $\mathcal{H}$ are well defined. Also, say that an $r$-uniform hypergraph $\mathcal{H}$ is \emph{strongly-algebraic of complexity $(n,d)$}, if there exists a single polynomial $f:(\mathbb{F}^{n})^{r}\rightarrow \mathbb{F}$ of degree at most $d$ such that $\{\mathbf{v}_1,\dots,\mathbf{v}_r\}\in V(\mathcal{H})^{(r)}$ is an edge if and only if $f(\mathbf{v}_1,\dots,\mathbf{v}_r)\neq 0$. We assume that the statement $f(\mathbf{v}_1,\dots,\mathbf{v}_r)=0$ is invariant under the permutation of $\mathbf{v}_1,\dots,\mathbf{v}_r$. (We could have also defined edges by $f(\mathbf{v}_1,\dots,\mathbf{v}_r)= 0$, but later it will be more convenient to work with this definition.)

\subsection{Ramsey properties of algebraic graphs}

One of the most well known explicit constructions of graphs having small cliques and independent sets is due to Frankl and Wilson \cite{FW81}. For a prime $p$, they consider the graph $G$, whose vertex set is the $p^2-1$ element subsets of $[n]$, and the sets $A$ and $B$ are joined by an edge if $|A\cap B|\equiv -1 \pmod p$. This graph is strongly-algebraic of complexity $(n,2)$, as we can identify each set with its characteristic vector over  $\mathbb{F}_p^{n}$, and two such vectors $\mathbf{v}$ and $\mathbf{w}$ are joined by an edge if $\langle \mathbf{v},\mathbf{w}\rangle\neq -1$. Note that $G$ is also algebraic of complexity $(p^2-1,1,(p^2-1)^2)$ over $\mathbb{R}$, as we can identify each set of size $p^2-1$ with a vector listing its elements, and whether the vectors $\mathbf{u},\mathbf{v}\in \mathbb{R}^{p^2-1}$ are joined by an edge can be decided by the $(p^2-1)^2$ linear equations $\mathbf{u}(i)=\mathbf{v}(j)$. The number of vertices of $G$ is $N=\binom{n}{p^2-1}$, and the celebrated Frankl-Wilson theorem on restricted intersections implies that $G$ has no clique or independent set of size larger than $\binom{n}{p-1}=O_p(N^{1/(p+1)})$. Choosing $n=p^3$, the largest clique and independent set in $G$ has size $2^{O(\sqrt{\log N\log\log N})}$. Another candidate for a graph with good Ramsey properties is the \emph{Paley-graph.}  If $p=1\,(\hspace{-0.24cm}\mod 4)$ is a prime, the Paley-graph of order $p$ is the graph whose vertex set is $\mathbb{F}_p$, and $x$ and $y$ are joined by an edge if $x+y$ is a quadratic residue, that is, $(x+y)^{(p-1)/2}\neq -1$. Such graphs are strongly-algebraic of complexity $(1,(p-1)/2)$, and it is widely believed that Paley-graphs have only polylogarithmic sized cliques or independent sets. For the best known explicit constructions of Ramsey graphs, see the recent works of Chattopadhyay, Zuckerman~\cite{CZ16} and Cohen~\cite{C19}.

The above two constructions of algebraic graphs have at least one large parameter in their complexity. In this paper we show that this is not a coincidence, and if an algebraic graph of complexity $(n,d,m)$ has good Ramsey properties, then at least one of the parameters $n,d,m$ must be large. In 2001, R\'onyai, Babai and Ganapathy \cite{RBG01} considered a bipartite version of this problem.  A \emph{bi-clique} in a graph $G$ is two disjoint sets $A,B\subset V(G)$ such that $|A|=|B|$ and every vertex in $A$ is joined to every vertex in $B$ by an edge. Note that if $G$ is the  random $N$-vertex graph with edge probability $1/2$, then the size of the largest bi-clique in both $G$ and its complement is $O(\log N)$, with high probability. In contrast,  R\'onyai, Babai and Ganapathy \cite{RBG01} proved that if $G$ is an algebraic graph on $N$ vertices of complexity $(n,d,m)$, then either $G$ or its complement contains a bi-clique of size at least $cN^{1/(n+1)}$, where $c=c(n,d,m)>0$. However, the existence of large bi-cliques or their complements does not imply the existence of large cliques or independent sets. Our first theorem extends the result of R\'onyai, Babai and Ganapathy as follows.

\begin{theorem}\label{thm:main}
There exists a constant $c>0$ such that the following holds. Let $n,d,m,N$ be positive integers. Let $G$ be an algebraic graph of complexity $(n,d,m)$ on $N$ vertices. Then $G$ contains either a clique or an independent set of size at least $c'N^{1/\gamma}$, where $c'=c'(n,d,m)>0$ and $$\gamma=c n m\min\left\{d,\frac{n\log d}{\log n}\right\}.$$
\end{theorem}

\noindent
I.e., the growth of exponent $\gamma$ as a function of $n$ is linear, while as a function of $d$, it is at most logarithmic. 

\bigskip

The  more general \emph{multicolor Ramsey problem} considers edge-colorings of the complete graph $K_{N}$ with $t$ colors, and asks how small can be the size of the largest monochromatic clique. A coloring $c:E(K_{N})\rightarrow [t]$ of the complete graph $K_{N}$ with $t$ colors is \emph{algebraic of complexity $(n,d,m)$}, if there exist $m$ polynomials $f_1,\dots,f_{m}:(\mathbb{F}^{n})^{2}\rightarrow \mathbb{F}$ of degree at most $d$ and a function $\phi:\{\mbox{false},\mbox{true}\}^{m}\rightarrow [t]$ such that $c(\mathbf{u},\mathbf{v})=\phi(\{f_{i}(\mathbf{u},\mathbf{v})=0\}_{i\in [m]})$.

In case $t=p+1$, where $p$ is prime, the best known construction (not necessarily explicit) of a coloring of $K_N$ with $t$ colors having small monochromatic cliques is a recent result of Conlon and Ferber \cite{CF21}, which was further improved for $t>3$ by Wigderson \cite{W20}. In \cite{CF21}, it is proved that if $n$ is a positive integer and $N=2^{n/2}p^{3n/8+o(n)}$, then $K_{N}$ has a coloring with $t$ colors containing no monochromatic clique of size $n$. Surprisingly, the coloring they provide is almost algebraic. More precisely, first they consider an algebraic $(t-1)$-coloring of complexity $(n,2,t-1)$, and then they recolor the last color class with two new colors randomly.

We prove the extension of Theorem \ref{thm:main}, which shows that in an algebraic coloring of complexity $(n,d,m)$, at least one of the parameters $n,d,m$ must be large if we want to avoid large monochromatic cliques.

\begin{theorem}\label{thm:multicolor}
There exists a constant $c>0$ such that the following holds. Let $n,d,m,t,N$ be positive integers. Every algebraic coloring of $K_{N}$ of complexity $(n,d,m)$ with $t$ colors contains a monochromatic clique of size at least $c'N^{1/\gamma}$, where $c'=c'(n,d,m)$ and $$\gamma=c n m\min\left\{d,\frac{n\log d}{\log n}\right\}.$$
\end{theorem}

\subsection{Ramsey properties of algebraic hypergraphs}

The Ramsey problem for hypergraphs is also subject of extensive study. For positive integers $r$ and $t$, let $R_{r}(t)$ denote the smallest $N$ such that any $r$-uniform hypergraph on $N$ vertices contains either a clique or an independent set of size $t$. Erd\H{o}s, Hajnal and Rado \cite{EHR65} and Erd\H{o}s and Rado \cite{ER52} proved that there exist constants $c,C>0$ such that
$$\mbox{tw}_{r-1}(ct^2)<R_{r}(t)<\mbox{tw}_{r}(Ct),$$
where the \emph{tower function} $\mbox{tw}_{k}(x)$ is defined as $\mbox{tw}_{1}(x)=x$ and $\mbox{tw}_{k}(x)=2^{\mbox{tw}_{k-1}(x)}$. For recent developments on the topic, see also the works of Conlon, Fox, and Sudakov \cite{CFS10,CFS13}.

Semi-algebraic graphs and hypergraphs are closely related to algebraic ones and were first studied by Alon, Pach, Pinchasi, Radoi\v ci\'c and Sharir \cite{APPRS}. An $r$-uniform hypergraph $\mathcal{H}$ is \emph{semi-algebraic} of complexity $(n,d,m)$ if $V(\mathcal{H})\subset \mathbb{R}^{n}$, and there exist $m$ polynomials $f_1,\dots,f_{m}:(\mathbb{R}^{n})^{r}\rightarrow\mathbb{R}$ and a Boolean formula $\phi$ such that $\{\mathbf{v}_1,\dots,\mathbf{v}_r\}\in V(\mathcal{H})^{(r)}$ is an edge if and only if $\phi(\{f_{i}(\mathbf{v}_1,\dots,\mathbf{v}_r)\geq 0\}_{i\in [m]})$ is true.

Let $R^{n,d,m}_{r}(t)$ denote the smallest $N$ such that any $r$-uniform semi-algebraic hypergraph of complexity $(n,d,m)$ contains either a clique or an independent set of size $t$. Conlon et al. \cite{CFPSS14} studied the Ramsey problem for semi-algebraic hypergraphs and  proved that there exists $c=c(r,n,m,d)>0$ and $C=C(r,n,m,d)>0$ such that 
$$\mbox{tw}_{r-1}(ct)\leq R_{r}^{n,d,m}(t)<\mbox{tw}_{r-1}(t^C).$$
In the special case of $n=1$, Bukh and Matou\v sek \cite{BM12} proved that if $\mathcal{H}$ is an $r$-uniform semi-algebraic hypergraph of complexity $(n,d,m)$ on $N$ vertices, then one can always find a clique or independent set of size at least $c'\log\log N$, where $c'$ may not only depend on the complexity, but on the defining polynomials $f_1,\dots,f_m$ as well.
 
Quite surprisingly, algebraic hypergraphs behave very differently. 
We show that such hypergraphs of constant complexity contain polynomial sized cliques or independent sets.
 
\begin{theorem}\label{thm:hypergraph}
Let $r,n,d,m$ be positive integers, then there exists $\gamma=\gamma(r,n,d,m)>0$ and $c=c(r,n,d,m)>0$ such that the following holds. If $\mathcal{H}$ is an $r$-uniform algebraic hypergraph of complexity $(n,d,m)$ on $N$ vertices, then $\mathcal{H}$ contains either a clique or an independent set of size at least $cN^{1/\gamma}$. One can choose $$\gamma=2r^2m\left(\binom{n+d}{d}+1\right).$$
\end{theorem}

\noindent
Let us remark that the existence of $\gamma$ (whose dependence on $n,d,m$ is not quite explicit) also follows from a model theoretic argument of Malliaris and Shelah, Theorem 3.5 in \cite{MS14} (see also \cite{CS18} for a shorter proof). They show that if $\mathcal{H}$ is a $p$-stable $r$-uniform hypergraph on $N$ vertices (see \cite{CS18} for definitions), then $\mathcal{H}$ contains either a clique or an independent set of size $N^{\gamma}$, where $\gamma=\gamma(r,p)>0$. They also remark that $r$-uniform algebraic hypergraphs of complexity $(n,d,m)$ are also $p$-stable for some $p=p(r,n,d,m)$.

Theorem \ref{thm:hypergraph}  might give the feeling that semi-algebraic graphs and hypergraphs have always stronger Ramsey-type properties than algebraic ones. In the concluding remarks, we will show that this is not always true.

\subsection{Erd\H os-Hajnal conjecture and graphs of bounded VC-dimension} 

Our results are also closely related to the celebrated Erd\H{o}s-Hajnal conjecture. It was proved by Erd\H{o}s and Hajnal \cite{EH89} that if $G$ is a graph on $N$ vertices which contains no induced copy of some fixed graph $H$, then $G$ contains a clique or an independent set of size at least $e^{c\sqrt{\log N}}$, where $c=c(H)>0$ only depends on $H$. They also proposed the conjecture that $G$ contains a clique or an independent set of size at least $N^{c}$ for some $c=c(H)>0$. This conjecture, referred to as the Erd\H{o}s-Hajnal conjecture, is one of the central open problems in graph theory. Say that a family of graphs $\mathcal{G}$ has the \emph{Erd\H{o}s-Hajnal property}, if there exists a constant $c=c(\mathcal{G})>0$ such that every $G\in \mathcal{G}$ contains a clique or an independent set of size at least $|V(G)|^{c}$.  The Erd\H{o}s-Hajnal conjecture is equivalent to the statement that every hereditary family of graphs has the Erd\H{o}s-Hajnal property, unless it is the family of all graphs. Here, we say that a family $\mathcal{F}$ of $r$-uniform hypergraphs is \emph{hereditary}, if $\mathcal{F}$ is closed under taking induced subhypergraphs.

Although the Erd\H{o}s-Hajnal conjecture is still wide open, its bipartite analog was solved twenty years ago by  
Erd\H{o}s, Hajnal and Pach \cite{EHP00}. They proved that if an $N$-vertex graph $G$ contains no induced copy of $H$, then either $G$ or its complement has a bi-clique of size at least $N^{c}$ for some $c=c(H)>0$. Fox and Sudakov \cite{FS09} improved this by showing that either $G$ contains a bi-clique of size $N^{c}$, or it contains an independent set of size at least $N^{c}$. This suggests that finding polynomial sized cliques or independent sets is considerably harder than finding bi-cliques, or their complements.

The aforementioned bound of Erd\H{o}s and Hajnal was recently improved for graphs of bounded VC-dimension. The concept of graphs of bounded VC-dimension extends both algebraic and semi-algebraic graphs of bounded complexity (as we shall see later). Let $\mathcal{F}$ be a family of subsets of a base set $X$. If $U\subset X$, then $\mathcal{F}\mid_{U}=\{A\cap U:A\in\mathcal{F}\}$ is the \emph{projection} of $\mathcal{F}$ to $U$. Also, for every positive integer $z$, the \emph{shatter function of $\mathcal{F}$} is defined as $$\pi_{\mathcal{F}}(z)=\max_{U\in X^{(z)}}|\mathcal{F}\mid_{U}|.$$ The \emph{VC-dimension (Vapnik-Chervonenkis dimension)} of the family $\mathcal{F}$ is the largest integer $n$ such that $\pi_{\mathcal{F}}(n)=2^n$, that is, there exists a set $U\subset X$ of size $n$ such that $\mathcal{F}\mid_{U}=2^{U}$. In this case, we say that $U$ is \emph{shattered} by $\mathcal{F}$. The VC-dimension, introduced by Vapnik and Chervonenkis \cite{VC71}, is one of the most widely used measures of complexity of set systems in computer science and computational geometry. We say that a graph $G$ has \emph{VC-dimension $n$} if the family $\{N(v):v\in V(G)\}$ has VC-dimension~$n$. Ramsey properties of graphs of bounded VC-dimension were recently studied by Fox, Pach and Suk \cite{FPS19}. They proved that if $G$ is a graph on $N$ vertices of VC-dimension $n$, then $G$ contains either a clique or an independent set of size $e^{(\log N)^{1-o(1)}}$, where $o(1)\rightarrow 0$ as $N\rightarrow \infty$ while $n$ is fixed. It is open whether this bound can be improved to $N^{c}$, where $c=c(n)>0$, that is, whether the family of graphs of VC-dimension at most $n$ has the Erd\H{o}s-Hajnal property. Theorem \ref{thm:main} shows that this holds for the family of algebraic graphs of constant complexity. 

\begin{corollary}
Let $n,d,m$ be positive integers. The family of algebraic graphs of complexity $(n,d,m)$ has the Erd\H{o}s-Hajnal property.
\end{corollary}

\subsection{Regularity lemma for algebraic hypergraphs}

Graph and hypergraph regularity lemmas are among the most powerful tools in combinatorics. Let $\mathcal{H}$ be an $r$-uniform hypergraph, and let $V_{1},\dots,V_{r}$ be disjoint subsets of $V(\mathcal{H})$. The \emph{density} of $(V_{1},\dots,V_{r})$ is $d(V_{1},\dots,V_{r})=\frac{|E(V_{1},\dots,V_{r})|}{|V_{1}|\dots|V_{r}|},$ where $E(V_1,\dots,V_{r})$ is the set of edges containing exactly one vertex from each of $V_{1},\dots,V_{r}$. A partition of $V(\mathcal{H})$ into parts $V_1,\dots,V_{K}$ is \emph{equitable} if the sizes of $V_1,\dots,V_K$ differ by at most 1.

 If $G$ is a graph, and $(U,V)$ is a pair of disjoint subsets of the vertex set, then $(U,V)$ is \emph{$\epsilon$-regular}, if $|d(U,V)-d(U',V')|<\epsilon$ for every $U'\subset U$, $V'\subset V$ satisfying $|U'|\geq\epsilon |U|$ and $|V'|\geq \epsilon |V|$. The regularity lemma of Szemer\'edi \cite{Sz78} states that if $G$ is a graph and $\epsilon>0$, then $G$ has an equitable partition into $K$ parts with $1/\epsilon<K<m(\epsilon)$, where $m(\epsilon)$ depends only on $\epsilon$, such that all but $\epsilon$ fraction of the pairs of parts are $\epsilon$-regular. Unfortunately, the dependence of $m(\epsilon)$ on $1/\epsilon$ is Ackermann-type, so often this regularity lemma is quite inefficient.
 
For certain special families of hypergraphs one can obtain stronger results. Say that the $r$-tuple of sets $(V_{1},\dots,V_{r})$ in the $r$-uniform hypergraph $\mathcal{H}$ is \emph{$\epsilon$-homogeneous} if $d(V_{1},\dots,V_{r})\leq \epsilon$ or $d(V_{1},\dots,V_{r})\geq 1-\epsilon$. Also, say that $(V_{1},\dots,V_{r})$ is \emph{homogeneous} if it is $0$-homogeneous. It was proved in the aforementioned paper of Fox, Pach and Suk \cite{FPS19} that if $\mathcal{H}$ has VC-dimension $n$, then $\mathcal{H}$ has an equitable partition into $K$ parts with $1/\epsilon<K<O_{r,n}((1/\epsilon)^{2n+1})$ such that all but at most $\epsilon$ fraction of the $r$-tuples of parts are $\epsilon$-homogeneous. The authors refer to this as ``ultra-strong regularity lemma'' and their bound on the number of parts has exponent which is optimal up to an absolute constant factor. This result
improves on a sequence of earlier regularity lemmas for this class of hypergraphs \cite{AFN07,CS,LSz07}. One can do even better for semi-algebraic graphs. Fox, Pach and Suk \cite{FPS16} also proved that if $\mathcal{H}$ is semi-algebraic of constant complexity, then $\mathcal{H}$ has an equitable partition into $K$ parts, where $K$ is bounded by a polynomial of $\epsilon$ depending only on the complexity, such that all but at most $\epsilon$ fraction of the $r$-tuples of parts are homogeneous.
 
 But what can we say about algebraic hypergraphs? It turns out that it is too much to ask for a partition in which all but a small fraction of $r$-tuples of parts are homogeneous. Indeed, if this would be true, it would imply that if $G$ is an algebraic graph of complexity $(n,d,m)$, then either $G$ or its complement contains a bi-clique of linear size. In the concluding remarks, we present an example showing that this does not hold in general. On the other hand, Chernikov and Starchenko \cite{CS} showed (see Theorem 4.13) that if a hypergraph $\mathcal{H}$ is $p$-stable, then $\mathcal{H}$ has an equitable partition into $K$ parts with $1/\epsilon<K<O_{p}((1/\epsilon)^{p+1})$ such that all $r$-tuples of parts are $\epsilon$-homogeneous. Here, we will prove the following interesting regularity type lemma for strongly-algebraic hypergraphs, which shows that strongly-algebraic hypergraphs are halfway between semi-algebraic hypergraphs, and hypergraphs of bounded VC-dimension.  

\begin{theorem}\label{thm:regularity}
Let $r,n,d$ be positive integers, then there exists $c=c(r,n,d)$ such that the following holds. Let $\mathcal{H}$ be a strongly-algebraic $r$-uniform hypergraph of complexity $(n,d)$. Then $V(\mathcal{H})$ has an equitable partition $V_1,\dots,V_{K}$ with $8/\epsilon<K<c(1/\epsilon)^{r!(2n+1)}$ parts such that all but at most $\epsilon$-fraction of the $r$-tuples of parts are either empty, or have density at least $1-\epsilon$.
\end{theorem}

\noindent
Another regularity lemma in algebraic setting was proved by Tao \cite{T05}. This lemma applies to graphs and hypergraphs whose vertices and edges are both defined by algebraic varieties of bounded complexity over some field $\mathbb{F}$ and more generally are definable sets of bounded complexity (see \cite{T05} for details). For example, the above mentioned, Paley-graph falls into this category. Tao's lemma shows that remarkably one can find an $\epsilon$-regular partition of the vertices of the graph or hypergraph in question into constant number of (definable) parts, where $\epsilon\approx |\mathbb{F}|^{-1/4}$. Interestingly, this seems to be quite different from our regularity lemma for strongly-algebraic hypergraphs. 

Let us present an application of Theorem \ref{thm:regularity} to the Ramsey problem for strongly-algebraic hypergraphs.
The following theorem roughly tells us that finding an independent set of size $s$ for some constant $s=s(r,n,d)$ in a strongly-algebraic $r$-uniform hypergraph of complexity $(n,d)$ is almost as `costly' as finding a very large independent set. This suggests that the exponent in Theorem \ref{thm:hypergraph} can be improved and in order to do so, it is enough to show that every such hypergraph contains either an independent set of size $s$, or a large clique.

\begin{theorem}\label{thm:hereditary}
Let $r,n,d$ be positive integers, $s=(r-1)\binom{n+d}{d}+1$, and $c,\alpha,\beta>0$. Let $\mathcal{F}$ be a hereditary family of strongly-algebraic $r$-uniform hypergraphs of complexity $(n,d)$. Suppose that each $\mathcal{H}\in\mathcal{F}$ on $N$ vertices contains either a clique of size at least $cN^{\alpha}$, or an independent set of size~$s$. Then every $\mathcal{H}\in\mathcal{F}$ on $N$ vertices contains either a clique of size $c_1N^{d\alpha\beta/n}$, or an independent set of size at least $c_2N^{1-\beta}$, where $c_1,c_2>0$ depend only on the parameters $r,n,d,c,\alpha,\beta$, and $d=d(r)>0$ depends only on $r$.
\end{theorem}

Our paper is organized as follows. In the next section, we introduce our notation, and we prepare several tools in order to prove our theorems. We also present the proof of Theorem \ref{thm:hypergraph} in this section. Then, in Section \ref{sect:graphs}, we present the proof of Theorems \ref{thm:main} and \ref{thm:multicolor}. We continue with the proof of Theorem \ref{thm:regularity} in Section \ref{sect:regularity}, and present the proof of Theorem \ref{thm:hereditary} in Section \ref{sect:hypergraph}. Finally, we provide some discussion in the Concluding remarks.

\section{Properties of algebraic hypergraphs}

In this section, we prepare some tools for the proofs of our main results, and we prove Theorem \ref{thm:hypergraph}. First, let us introduce the notation we use throughout the paper, which is mostly conventional.

\subsection{Notation and preliminaries}

If $V$ is some set and $s$ is a positive integer, $V^{(s)}$ denotes the family of $s$ element subsets of $V$. Let $\mathcal{H}$ be an $r$-uniform hypergraph.  The \emph{density} of $\mathcal{H}$ is $d(\mathcal{H})=|E(\mathcal{H})|/\binom{|V(\mathcal{H})|}{r}$. If $X\in V(\mathcal{H})^{(s)}$ for some $1\leq s\leq r-1$, then $$N_{\mathcal{H}}(X)=N(X)=\{Y\in V(\mathcal{H})^{(r-s)}: X\cup Y\in E(\mathcal{H})\}$$ is the neighborhood of $X$. 

The following is a well known result, which tells us that hypergraphs of density very close to 1 contain large cliques.

\begin{lemma}\label{lemma:clique}
Let $N$ be a positive integer, $\frac{1}{N^{r-1}}<\alpha<\frac{1}{2}$, and let $\mathcal{H}$ be an $r$-uniform hypergraph on $N$ vertices of density at least $1-\alpha$. Then $\mathcal{H}$ contains a clique of size at least $\frac{1}{4}(1/\alpha)^{1/(r-1)}$. 
\end{lemma}

\begin{proof}
Select each vertex of $\mathcal{H}$ with probability $p=(2N\alpha^{1/(r-1)})^{-1}<1$, and let $U$ be the set of selected vertices. Let $X$ be the number of non-edges spanned by $U$, then $\mathbb{E}(X)\leq \alpha p^{r}\binom{N}{r}$. Delete a vertex of each non-edge in $\mathcal{H}[U]$, and let $V$ be the resulting set, then $V$ is a clique. We have
$$\mathbb{E}(|V|)\geq \mathbb{E}(|U|-X)=pN-\alpha p^{r}\binom{N}{r}\geq \frac{pN}{2}\geq \frac{1}{4}\left(\frac{1}{\alpha}\right)^{1/(r-1)},$$
where the last two inequalities hold by the choice of $p$. Hence, there exists a choice for $U$ such that $|V|\geq\frac{1}{4}(1/\alpha)^{1/(r-1)}$.
\end{proof}

In order to describe hypergraphs defined by a single polynomial, let us introduce directed hypergraphs. A \emph{directed $r$-uniform hypergraph (or $r$-uniform dihypergraph)} is a pair $\mathcal{H}=(V,E)$, where $V$ is the set of vertices, and $E$ is a set of (ordered) $r$-tuples of distinct elements of $V$, called edges. Given an $r$-element subset $f$ of $V$, an \emph{orientation of $f$} is an $r$-tuple containing the elements of $f$. In a directed $r$-uniform hypergraph, we allow multiple orientations of the same $r$-element set. Say that an $r$-element set $X$ of a dihypergraph $\mathcal{H}$ is \emph{complete} if all $r!$ orientations of $X$ are edges. Let  $[\mathcal{H}]$ be the hypergraph formed by the complete edges of $\mathcal{H}$. A \emph{clique} in an $r$-uniform directed hypergraph $\mathcal{H}$ is a clique in $[\mathcal{H}]$. Also, an \emph{independent set} of $\mathcal{H}$ is a subset of vertices in which no $r$-tuple forms an edge.  

Given $I\subset [r]$ and an $|I|$-tuple of vertices $X=(v_i)_{i\in I}$, the \emph{$I$-neighborhood} of $X$ is  $N_{I}(X)=N_{\mathcal{H},I}(X)=\{(v_j)_{j\in [r]\setminus I}\in V^{r-|I|}:(v_1,\dots,v_r)\in E\}$. Note that the vertices of $X$ appear in the corresponding directed edges in the order given by $I$. In the case of directed graphs, we write simply $N^{+}(v)$ and $N^{-}(v)$ for the out- and in-neighborhood of $v$, respectively. If $X$ is an $r$-tuple and $k\in [r]$, then $X(k)$ is the $k$-th element of $X$, and $\hat{X}(k)=(X(1),\dots,X(k-1),X(k+1),\dots,X(r)).$

\subsection{Linear algebra}

In this section, we collect some facts from linear algebra and we introduce the flattening rank of tensors. Let $T:A_1\times\dots\times A_r\rightarrow \mathbb{F}$ be an $r$-dimensional tensor, where $A_1,\dots,A_r$ are finite sets, and $\mathbb{F}$ is a field. For $i\in [r]$, the \emph{$i$-flattening rank} of $T$, denoted by $\frank_{i}(T)$, is defined as follows. Let $B_{i}=A_1\times\dots\times A_{i-1}\times A_{i+1}\times\dots\times A_{r}$, then $T$ can be viewed as a matrix $M$ with rows indexed by $A_{i}$, and columns indexed by $B_{i}$. Then $\frank_{i}(T):=\mbox{rank}(M)$. Note that $\frank_{i}(T)=1$ if and only if $T\neq 0$, and there exist two functions $f:A_{i}\rightarrow \mathbb{F}$ and $g:B_i\rightarrow \mathbb{F}$ such that $T(a_1,\dots,a_r)=f(a_i)g(a_1,\dots,a_{i-1},a_{i+1},\dots,a_{r})$. Also, the $i$-flattening rank of $T$ is the minimum $t$ such that $T$ is the sum of $t$ tensors of $i$-flattening rank 1. Equivalently, $\frank_{i}(T)$ is the dimension of the vector space generated by the rows of $T$ in the $i$-th dimension.

It is easy to see that the $i$-flattening rank satisfies the usual properties of rank. It is subadditive, and if $T'$ is a subtensor of $T$, then $\frank_i(T')\leq \frank_i(T)$. Here, $T':A_1'\times\dots\times A_{r}'\rightarrow\mathbb{F}$ is a \emph{subtensor} of $T:A_1\times\dots\times A_{r}\rightarrow\mathbb{F}$ if $A_{i}'\subset A_{i}$ for $i\in [r]$, and $T'(a_1,\dots,a_r)=T(a_1,\dots,a_r)$ for $(a_1,\dots,a_{r})\in A_{1}'\times\dots\times A_{r}'$. A detailed discussion of $i$-flattening rank and its combinatorial applications can be found in \cite{frank}. In this paper, we use the following property of the $i$-flattening rank of tensors defined by polynomials. 

\begin{lemma}\label{lemma:tensorrank}
Let $f:(\mathbb{F}^{n})^{r}\rightarrow\mathbb{F}$ be a polynomial of degree at most $d$, and let $V\subset \mathbb{F}^{n}$. Define the $r$-dimensional tensor $T:V^{r}\rightarrow \mathbb{F}$ such that $T(X):=f(X)$ for $X\in V^{r}$. Then $\frank_{i}(T)\leq \binom{n+d}{d}$ for $i\in[r]$.
\end{lemma}

\begin{proof}
 Let $\Lambda=\{\mathbf{\alpha}\in \mathbb{N}^{n}:\sum_{i=1}^{n}\mathbf{\alpha}(i)\leq d\}$. For $\mathbf{x}\in \mathbb{F}^{n}$ and $\alpha\in \mathbb{N}^{n}$, let $h_{\alpha}(\mathbf{x})=\mathbf{x}(1)^{\mathbf{\alpha}(1)}\dots \mathbf{x}(n)^{\mathbf{\alpha}(n)}$.
Then the polynomial $f$ can be written as $$f(X)=\sum_{\alpha\in\Lambda} h_{\alpha}(X(i))g_{\mathbf{\alpha}}(\hat{X}(i)),$$
where $X\in V^{r}$ is considered as an $r$-tuple of elements of $V$, and $g_{\mathbf{\alpha}}:\mathbb{F}^{(r-1)n}\rightarrow\mathbb{F}$ is some polynomial for $\alpha\in\Lambda$. But then $T$ is the sum of $|\Lambda|\leq\binom{n+d}{d}$ tensors of $i$-flattening rank 1, so $\frank_i(T)\leq \binom{n+d}{d}$.
\end{proof}

We use this lemma in the next section to define a family of forbidden directed subhypergraphs in dihypergraphs defined by a single polynomial. Working with the flattening rank (versus matrix rank) is not necessary for this application, but later (see Section \ref{sect:hypergraph}) we will exploit more of its properties.

\subsection{Forbidden subhypergraphs}

In this section, we show that algebraic dihypergraphs defined by a single polynomial avoid the following simple family of dihypergraphs.

\begin{definition}
\normalfont For positive integers $r,s,k$ with $k\in [r]$, let $\mathcal{M}(r,s,k)$ be the family of $r$-uniform dihypergraphs $\mathcal{M}$ having the following form. There are given, not all necessarily distinct, vertices $u_{i,j}$ for $(i,j)\in [s]\times [r]$ forming the vertex set of $\mathcal{M}$. For $(i,i')\in [s]\times [s]$, let $X_{i,i'}$ be the $r$-tuple (of distinct) vertices satisfying $X_{i,i'}(j)=u_{i,j}$ for $j\in [r]\setminus\{k\}$, and $X_{i,i'}(k)=u_{i',k}$. Then  $X_{i,i}$ is an edge of $\mathcal{M}$ for $i\in [s]$, and $X_{i,i'}$ is not an edge for $1\leq i<i'\leq s$. The rest of the $r$-tuples can be either edges or non-edges. Also, let $\mathcal{M}(r,s)$ be the union of the families $\mathcal{M}(r,s,k)$ for $k\in [r]$.
\end{definition}

\begin{lemma}\label{lemma:matching}
Let $\mathcal{H}$ be an algebraic $r$-uniform dihypergraph of complexity $(n,d,1)$. Then $\mathcal{H}$ contains no member of $\mathcal{M}(r,s)$ for $s>\binom{n+d}{d}$.
\end{lemma}

\begin{proof}
Let $V=V(\mathcal{H})$, and let $f$ be the polynomial defining $\mathcal{H}$. Define the tensor $T:V^{r}\rightarrow \mathbb{F}$ 
such that for $X\in V^{r}$, $T(X):=f(X)$. For $k\in [r]$, let $M_{k}:V\times V^{r-1}\rightarrow \mathbb{F}$ be the matrix defined as $M_{k}(X(k),\hat{X}(k)):=T(X)$ for $X\in V^{r}$. Then by Lemma \ref{lemma:tensorrank}, we have
$\mbox{rank}(M_{k})=\frank_{k}(T)\leq\binom{n+d}{d}$. 

Let $s>\binom{n+d}{n}$, and suppose that $\mathcal{H}$ contains a copy of a member of $\mathcal{M}(r,s,k)$ for some $k\in [r]$. Then there exist $\mathbf{u}_{i,j}\in V$ for $(i,j)\in [s]\times [r]$ such that the following holds. For $(i,i')\in [s]\times [s]$, define $U_{i,i'}\in V^{r}$ such that $U_{i,i'}(j)=\mathbf{u}_{i,j}$ for $j\neq k$, and $U_{i,i'}(k)=\mathbf{u}_{i',k}$. Then $f(U_{i,i})\neq 0$ for $i\in [s]$ and $f(U_{i,i'})=0$ if $1\leq i< i'\leq s$. But then the submatrix of $M_{k}$ induced by the rows $\mathbf{u}_{1,k},\dots,\mathbf{u}_{s,k}$ and columns $\hat{U}_{1,1}(k),\dots,\hat{U}_{s,s}(k)$ is an upper-triangular matrix, which has full rank. Therefore, $\mbox{rank}(A)\geq s$, which is a contradiction. 
\end{proof}


\begin{lemma}\label{lemma:weak}
Let $\mathcal{H}_1,\dots,\mathcal{H}_m$ be $r$-uniform dihypergraphs on an $N$ element vertex set $V$. If $\mathcal{H}_{i}$ contains no member of $\mathcal{M}(r,s)$ for $i\in [m]$, then there exists $U\subset V$ such that $\mathcal{H}_{i}[U]$ is either a clique or an independent set for $i\in [m]$, and $|U|\geq cN^{1/2r^2ms}$, where $c=c(r,s,m)>0$ depends only on $r,s,m$.
\end{lemma}

\begin{proof}

We prepare the proof with the following claims.

\begin{claim}\label{claim:med_degree}
Let $\mathcal{G}$ be a nonempty $r$-uniform hypergraph on $M\geq 100r$ vertices, and suppose that the density of $\mathcal{G}$ is at most $1-\alpha$. Then there exists $X\in V(\mathcal{G})^{(r-1)}$ such that $1\leq |N(X)|\leq (1-\alpha/2r)M$.
\end{claim}

\begin{proof}
 Let $\alpha'$ be the unique real number such that $|E(\mathcal{G})|=\binom{(1-\alpha')M}{r} \leq (1-\alpha)\binom{M}{r}$, then one can check that $\alpha'\geq \alpha/2r$. We prove by induction on $r$ that $\mathcal{G}$ contains an $(r-1)$-element set $X$ such that $1\leq |N(X)|\leq (1-\alpha')M$. This is certainly true if $r=1$, so suppose that $r\geq 2$.

Let $U\subset V(\mathcal{G})$ be the set of vertices with at least 1 neighbour. Then there exists $u\in U$ such that 
$$|N(u)|\leq \frac{r|E(\mathcal{G})|}{|U|}.$$ Note that $\binom{|U|}{r}\geq |E(\mathcal{G})|$ as every edge is contained in $U$, so $|U|\geq (1-\alpha')M$. But then 
$$|N(u)|\leq \frac{r\binom{(1-\alpha')M}{r}}{(1-\alpha')M}=\binom{(1-\alpha')M-1}{r-1}<\binom{(1-\alpha')M}{r-1}.$$
Let $\mathcal{G}'$ be the link graph of $u$. Then $\mathcal{G}'$ is a nonempty $(r-1)$-uniform hypergraph such that $|E(\mathcal{G}')|\leq \binom{(1-\alpha')M}{r-1}$. Therefore, by our induction hypothesis, there exists an $(r-2)$-element set $X'$ such that $1\leq |N_{\mathcal{G}'}(X')|\leq (1-\alpha')M$. But then $X=X'\cup\{u\}$ satisfies  $1\leq |N_{\mathcal{G}}(X)|\leq (1-\alpha')M$ as well, finishing the proof.
\end{proof}

\begin{claim}\label{claim:med_degree_dir}
Let $\mathcal{G}$ be a nonempty $r$-uniform dihypergraph on $M\geq 100r$ vertices, and suppose that the density of $[\mathcal{G}]$ is at most $1-\alpha$. Then there exist $\ell\in [r]$ and $Y\in V(\mathcal{G})^{r-1}$ such that $$1\leq |N_{[r]\setminus\{\ell\}}(Y)|\leq \left(1-\frac{\alpha}{2r\cdot r!}\right)M.$$
\end{claim}

\begin{proof}
First, consider the case that $[\mathcal{G}]$ is nonempty. Then by the previous claim, there  exists $X\in V(\mathcal{G})^{(r-1)}$ such that $1\leq |N_{[\mathcal{G}]}(X)|\leq (1-\alpha/2r)M$. Let $U=V(\mathcal{G})\setminus N_{[\mathcal{G}]}(X)$, and note that $|U|\leq \frac{\alpha}{2r}M$. For every $u\in U$, the set $X\cup\{u\}$ has an orientation not present in $\mathcal{G}$. Hence, there exist $U'\subset U$, $\ell\in [r]$, and an orientation $Y$ of $X$ such that $|U'|\geq |U|/r!$, and $U'$ is disjoint from $N_{[r]\setminus\{\ell\}}(Y)$. But then  $|N_{[r]\setminus\{\ell\}}(Y)|\leq (1-\frac{\alpha}{2r\cdot r!})M$. Also, $ N_{[\mathcal{G}]}(Y)\subset N_{[r]\setminus\{\ell\}}(Y)$, so $|N_{[r]\setminus\{\ell\}}(Y)|\geq 1$ as well. 

Now consider the case when $[\mathcal{G}]$ is empty. Let $\beta=1/r!$. In this case we show that there exist $\ell\in [r]$ and $Y\in V(\mathcal{G})^{r-1}$ such that $1\leq |N_{[r]\setminus\{\ell\}}(Y)|\leq (1-\beta)M$. Suppose this is not the case, that is, for each $W\in E(\mathcal{G})$ and $\ell\in [r]$ we have $|N_{[r]\setminus \{\ell\}}(\hat{W}(\ell))|> (1-\beta)M$. We define the sets of edges $F_0,\dots,F_r$ of $\mathcal{G}$ as follows. Let $F_0$ contain a single edge $W_0$ of $\mathcal{G}$. For $i\in [r]$, $F_i$ will have the following form. There exist $x_1,\dots,x_i\in V(\mathcal{G})$ such that $F_i$ contains all $r$-tuples $W$ for which $W$ agrees with $W_0$ in exactly $(r-i)$ coordinates, and the rest of the coordinates are a permutation of $x_1,\dots,x_i$. Note that then $|F_i|=r(r-1)\dots(r-i+1)$. If $F_i$ is already defined satisfying this property, we define $F_{i+1}$ as follows. For each $W\in F_i$, let $I_W$ be the set of coordinates $\ell\in [r]$ for which $W(\ell)=W_0(\ell)$. Consider the set $$U=\bigcap_{W\in F_i}\bigcap_{\ell\in I_W} N_{[r]\setminus \{\ell\}}(\hat{W}(\ell)).$$
Then $U$ is the intersection of $(r-i)|F_i|\leq r!$ subsets of $V(\mathcal{G})$ of size at least $(1-\beta)M$, hence, $U$ is nonempty. Therefore, there exists $x_{i+1}\in U$. Let $F_{i+1}$ be the set of all edges $W'$ for which $\hat{W}'(\ell)=\hat{W}(\ell)$ and $W'(\ell)=x_{i+1}$ for some $W\in F_i$ and $\ell\in I_W$. Then $F_{i+1}$ has the desired properties. In particular, $F_{r}$ contains every orientation of $\{x_1,\dots,x_r\}$ contradicting that $[\mathcal{G}]$ is nonempty. This finishes the proof.

Let us illustrate the previous argument in case $r=3$. Here, $F_0$ contains some edge $abc$. Then, we can find $x=x_1$ such that $abx,axc,xbc$ are all edges, they form $F_1$. Then, we can find some $y=x_2$ such that $ayx,ybx,axy,yxc,xby,xyc$ are all edges, they form $F_2$. Finally, we can find some $z=x_3$ such that $zyx,yzx,zxy,yxz,xzy,xyz$ are all edges, forming $F_3$.
\end{proof}

Let $\alpha=2r\cdot r! N^{-1/2rms}$. Let $U_{0}=V(\mathcal{H})$, $s_{0,i,k}=s$ for $(i,k)\in [m]\times [r]$, and define the nested sequence of vertex sets $U_{0}\supset U_{1}\supset\dots$ and sequence of integers $s_{\ell,i,k}$ for $\ell=1,\dots$ and $(i,k)\in [m]\times[r]$ as follows. Suppose that $U_{\ell}$ and $s_{\ell,i,k}$ are already defined for some $\ell\geq 0$ such that $\mathcal{H}_{i}[U_{\ell}]$ contains no member of $\mathcal{M}(r,s_{\ell,i,k},k)$ for $(i,k)\in [m]\times [r]$.
Let $I\subset [m]$ be the set of indices $i$ such that $\mathcal{H}_{i}[U_{\ell}]$ is nonempty. For such indices $i$, we must have $s_{\ell,i}>1$, as the family  $\mathcal{M}(r,1)$ is composed of a single dihypergraph containing a single edge. Consider three cases.

\begin{description}
  \item[Case 1.] $I=\emptyset$. In this case we stop and set $U=U_{\ell}$.
    
    \item[Case 2.] For every $i\in I$, the density of $[\mathcal{H}_{i}][U_\ell]$ is at least  $1-\alpha$. Let $\mathcal{H}_{I}=\bigcap_{i\in I} [\mathcal{H}_{i}]$, then $d(\mathcal{H}_{I}[U_{\ell}])\geq 1-\alpha|I|\geq 1-\alpha m$. Therefore, by Lemma \ref{lemma:clique}, $\mathcal{H}_{I}[U_{\ell}]$ contains a clique $U$ of size at least $\frac{1}{4}(1/\alpha m)^{1/(r-1)}$. Note that $\mathcal{H}_{i}[U]$ is a clique for every $i\in I$, and $\mathcal{H}_{i}[U]$ is an independent set for every $i\in [m]\setminus I$.
    
    \item[Case 3.] There exists $i\in I$ such that the density of  $[\mathcal{H}_{i}][U_{\ell}]$ is less than $1-\alpha$. As $\mathcal{H}_{i}[U_{\ell}]$ is not empty, we get by Claim \ref{claim:med_degree_dir} that there exist $X\subset U_{\ell}^{r-1}$ and $k\in [r]$ such that 
    $$1\leq |N_{\mathcal{H}_{i}[U_{\ell}],[r]\setminus\{k\}}(X)|\leq \left(1-\frac{\alpha}{2r\cdot r!}\right)|U_{\ell}|.$$
    Let $U_{\ell+1}=U_{\ell}\setminus N_{\mathcal{H}_{i}[U_{\ell}],[r]\setminus\{k\}}(X)$. Also, set $s_{\ell,i',k'}=s_{\ell+1,i',k'}$ if $(i',k')\neq (i,k)$, and let $s_{\ell+1,i,k}=s_{\ell,i,k}-1$. Note that $|U_{\ell+1}|\geq \frac{\alpha}{2r\cdot r!}|U_{\ell}|$, and $\mathcal{H}_{j}[U_{\ell}]$ contains no member of $\mathcal{M}(r,s_{\ell+1,i',k'},k)$ for $(i',k')\in [m]\times [r]$. The latter is clear if $(i',k')\neq (i,k)$. If $(i',k')=(i,k)$, let $f$ be any edge in $\mathcal{H}_{i}[U_{\ell}]$ with $\hat{f}(k)=X$. If $\mathcal{H}_{i}[U_{\ell+1}]$ contains a member of $\mathcal{M}(r,s_{\ell+1,i,k},k)$, then together with $f$, this forms a member of $\mathcal{M}(r,s_{\ell,i,k},k)$ in $\mathcal{H}_{i}[U_{\ell}]$, contradiction.
\end{description}

Let $L$ be the index $\ell$ for which we stop. For $\ell=1,\dots,L$, let $s_{\ell}=\sum_{i=1}^{m}\sum_{k=1}^{r}s_{\ell,i,k}$. Then $s_{0}=rms$, $s_{L}\geq rm$, and $s_{\ell+1}=s_{\ell}-1$, so we get $L\leq rm(s-1)$. Also, $|U_{\ell+1}|\geq \frac{\alpha}{2r}|U_{\ell}|$, so $|U_{L}|\geq (\alpha/2r\cdot r!)^{L}N\geq (\alpha/2r\cdot r!)^{rms}N\geq N^{1/2}$.

If we stop in Case 1., then $U=U_{\ell}$ and so $|U|\geq N^{1/2}$. If we stop in Case 2., then $|U|\geq \frac{1}{4}(1/\alpha m)^{1/(r-1)}\geq c N^{1/2msr(r-1)}$, where $c=c(r,s,m)>0$ depends only on $r,s,m$. This finishes the proof.
\end{proof}

From this, we can immediately deduce Theorem \ref{thm:hypergraph}.

\begin{proof}[Proof of Theorem \ref{thm:hypergraph}]
Let $f_1,\dots,f_m:(\mathbb{F}^{n})^r$ be polynomials of degree at most $d$ and $\phi$ be a Boolean formula defining $\mathcal{H}$, and let $\mathcal{H}_i$ be the algebraic dihypergraph defined by $f_i$ on $V(\mathcal{H})$ for $i\in [m]$. By Lemma \ref{lemma:matching}, $\mathcal{H}_{i}$ contains no member of $\mathcal{M}(r,s)$ for $s=\binom{n+d}{n}+1$. But then by Lemma \ref{lemma:weak}, there exists $U\subset V(\mathcal{H})$ such that $|U|\geq cN^{1/2r^2ms}$ and $\mathcal{H}_{i}[U]$ is either a clique or an independent set for $i\in [m]$. 

If $X\in V(\mathcal{H})^{(r)}$ and $Y$ is an orientation of $X$, then $X$ is an edge of $\mathcal{H}$ if and only if 
$$\phi([Y\not\in E(\mathcal{H}_{i})]_{i\in [m]})=\mbox{true}.$$
By the definition of $U$, the left-hand side is the same for every $Y\in U^{r}$, so $\mathcal{H}[U]$ is either a clique or an independent set. 
\end{proof}

Let us remark that the bound in Lemma \ref{lemma:weak} is optimal in the following sense. Let $G$ be a graph, which we view as a digraph in which each edge is oriented both ways. Note that if  $G$ contains no two sets of size $\lfloor s/2\rfloor$ with no edges between them, then $G$ contains no member of $\mathcal{M}(2,s)$ as well. But standard probabilistic arguments show that there are graphs $G$ on $N$ vertices containing no two such sets, and every clique and independent set in $G$ is of size $N^{O(1/s)}$. In the next section, we show that this bound can be significantly improved if we exploit some other properties of algebraic graphs as well.

\subsection{Zero-patterns}
One of the key results in our proof is the following lemma of  R\'onyai, Babai and Ganapathy \cite{RBG01} on the number of zero-pattern of polynomials.

Given a sequence of $m$ polynomials $\mathbf{f}=(f_1,\dots,f_{m})$ in $n$ variables over the field $\mathbb{F}$, a \emph{zero-pattern} of $\mathbf{f}$ is a sequence $\epsilon\in \{0,*\}^{m}$ for which there exists $x\in\mathbb{F}^{n}$ such that $f_{i}(x)=0$ if and only if $\epsilon(i)=0$. Let $Z(\mathbf{f})$ denote the number of zero-patterns of $\mathbf{f}$. Using elegant algebraic tools, R\'onyai, Babai and Ganapathy \cite{RBG01} proved the following.

\begin{lemma}\label{lemma:zeropatterns}
Let $f_1,\dots,f_m:\mathbb{F}^{n}\rightarrow \mathbb{F}$ be a sequence of polynomials of degree at most $d$. Then the number of zero-patterns of $\mathbf{f}=(f_1,\dots,f_m)$ is at most $\binom{md+n}{n}$. In particular, $Z(\mathbf{f})\leq cm^{n}$, where $c=c(n,d)$ depends only on $n$ and $d$.
\end{lemma}

\subsection{Weak-VC-dimension}

In this section, we show that algebraic graphs of complexity $(n,d,m)$ behave similarly as graphs of VC-dimension $n$. By the celebrated Sauer-Shelah lemma \cite{Sa72,Sh72}, if a family of sets $\mathcal{F}$ has VC-dimesnion $n$, then 
$$\pi_{\mathcal{F}}(z)\leq\sum_{i=0}^{n}\binom{z}{i}\leq cz^{n},$$
where $c=c(n)$ depends only on $n$. 

Let us relax the notion of VC-dimension as follows, and note that we now work with undirected hypergraphs. Say that the family $\mathcal{F}$ has \emph{weak-VC-dimension} $(c,n)$, if $\pi_{\mathcal{F}}(z)\leq cz^{n}$ for every positive integer $z$. Also, an $r$-uniform hypergraph $\mathcal{H}$ has \emph{weak-VC-dimension $(c,n)$} if the family $\{N(v):v\in V(\mathcal{H})\}\subset 2^{V(\mathcal{H})^{(r-1)}}$ has weak-VC-dimension $(c,n)$ (that is, $\{N(v):v\in V(\mathcal{H})\}$ is viewed as a family of subsets of the base set $\mathcal{V}=V(\mathcal{H})^{(r-1)}$). Also, say that a family of hypergraphs $\mathcal{F}$ has \emph{weak-VC-dimension $n$} if there exists a constant $c=c(\mathcal{F})$ such that every $\mathcal{H}\in\mathcal{F}$ has weak-VC-dimension $(c,n)$.

\begin{lemma}\label{lemma:weakvc}
Let $\mathcal{H}$ be an $r$-uniform algebraic hypergraph of complexity $(n,d,m)$, and let $\mathcal{F}=\{N(v):v\in V(\mathcal{H})\}$. Then 
$$\pi_{\mathcal{F}}(z)\leq \binom{zmd+n}{n}\leq cz^{n},$$
where $c=c(n,d,m)$ depends only on $n,d,m$.
\end{lemma}
\begin{proof}
Let $f_1,\dots,f_m:(\mathbb{F}^{n})^{r}\rightarrow \mathbb{F}$ be polynomials of degree at most $d$, whose zero-patterns determine the edges of $\mathcal{H}$. 

Let $\mathcal{V}=V(\mathcal{H})^{(r-1)}$, that is the base set of $\mathcal{F}$. If $\mathcal{U}\subset \mathcal{V}$ such that $|\mathcal{U}|=z$, then each element of the family $\{N(v)\cap \mathcal{U}:v\in V(\mathcal{H})\}$ is determined by a zero-pattern of the $zm$ polynomials $f_{i,U}:\mathbb{F}^{n}\rightarrow \mathbb{F}$ defined as $f_{i,U}(\mathbf{x})=f_{i}(U',\mathbf{x})$, where $i\in [m]$ and $U'\in (\mathbb{F}^{n})^{r-1}$ is a fixed orientation of the $(r-1)$-element set $U\in \mathcal{U}$. Therefore, by Lemma  \ref{lemma:zeropatterns}, we have $$\pi_{\mathcal{F}}(z)\leq \binom{zmd+n}{n}.$$
\end{proof}

Solving the inequality $\binom{zmd+n}{n}<2^z$, Lemma \ref{lemma:weakvc} shows that an algebraic graph of complexity $(n,d,m)$ has VC-dimension at most $n\log (ndm)$. Also, the VC-dimension of such graphs cannot be bounded by a function of $n$ alone: every graph is as an algebraic graph of complexity $(1,d,1)$ and $(1,1,m)$ for some positive integers $d$ and $m$. However, Lemma \ref{lemma:weakvc} also tells us that the family of algebraic graphs of complexity $(n,d,m)$ has weak-VC-dimension at most $n$.

Let $\delta>0$. Say that the family $\mathcal{F}$ of subsets of the base set $X$ is \emph{$\delta$-separated} if $|A\Delta B|\geq \delta|X|$ holds for any two distinct $A,B\in \mathcal{F}$. The following packing lemma is due to Haussler \cite{H95}. While it is originally stated for families of bounded VC-dimension, the same proof implies the following result as well.

\begin{lemma}\label{lemma:packing}
Let $n$ be a positive integer and $c>0$, then there exists $c'=c'(c,n)$ such that the following holds. Let $\mathcal{F}$ be a $\delta$-separated family of weak-VC-dimension $(c,n)$. Then $|\mathcal{F}|\leq c'(1/\delta)^{n}$.
\end{lemma}

Next, we show that Lemma \ref{lemma:weakvc} and Lemma \ref{lemma:packing} can be combined to prove that if $G$ is a graph of bounded weak-VC-dimension, whose density is not too close to 1, then $G$ contains two large subsets with few edges between them. This lemma is inspired by Theorem 1.3 in the paper of Fox, Pach and Suk \cite{FPS19}, which is a regularity-type result for hypergraphs of bounded VC-dimension. Later, we will also make use of this theorem, see Lemma \ref{lemma:ultrastrong} for the precise statement.

\begin{lemma}\label{lemma:sparsesubgraph}
	Let $n$ be a positive integer and $c>0$, then there exists $c_{0}=c_{0}(c,n)>0$ such that the following holds. Let $\alpha,\beta>0$, and let $G$ be a graph of weak-VC-dimension $(c,n)$ on $N$ vertices such that $d(G)\leq 1-\alpha$. Then $G$ contains two disjoint sets $A$ and $B$ such that $|A|\geq \frac{\alpha}{4}N $, $|B|\geq c_0\alpha(\alpha\beta)^{n} N$, and every vertex in $B$ has at most $\beta |A|$ neighbors in $A$.
\end{lemma}

\begin{proof}
	Let $U\subset V(G)$ be the set of vertices $v$ such that $|N(v)|\leq (1-\alpha/2)N$. Then $$(1-\alpha)\binom{N}{2}\geq |E(G)|\geq \frac{(N-|U|)}{2}\left(1-\frac{\alpha}{2}\right)N,$$
	from which we get 
	$$|U|> \frac{\alpha}{2} N.$$
	
	Let $\gamma=\alpha\beta/4$, and let $Z\subset U$ be maximal such that the family $\{N(v):v\in Z\}$ is $\gamma$ separated. Then by Lemma \ref{lemma:packing}, we have 
	$|Z|\leq c'(1/\gamma)^{n}$, where $c'=c'(c,n)$ depends only on the parameters $c$ and $n$. By the maximality of $Z$, for every $v\in U$ there exists $z\in Z$ such that $|N(v)\Delta N(z)|\leq \gamma N$. Hence, we can find $z_{0}\in Z$ such that the set $B$ formed by the vertices  $v$ such that $|N(v)\Delta N(z_0)|\leq \gamma N$ satisfies
	$$|B|\geq \frac{|U|}{|Z|}\geq \frac{\alpha\gamma^{n}}{2c'}N.$$ 
	By deleting some elements of $B$, we can assume that $|B|\leq \frac{\alpha}{4}$, and that $|B|$ still satisfies the previous inequality. Set $A=V(G)\setminus (B\cup N(z_{0}))$, then we show that $A$ and $B$ suffices.
	
	Indeed, $$|A|\geq N-|B|-|N(z_{0})|\geq N-\frac{\alpha}{4}N-\left(1-\frac{\alpha}{2}\right)N\geq \frac{\alpha}{4}N.$$
	Also, every $v\in B$ satisfies $|N(v)\Delta N(z_{0})|\leq \gamma N$, and $z_{0}$ has no neighbor in $A$, so $|N(v)\cap A|\leq \gamma N\leq \beta |A|$.
\end{proof}

\section{A Ramsey-type result for algebraic graphs}\label{sect:graphs}

In this section, we prove Theorems \ref{thm:main} and \ref{thm:multicolor}. If $G$ is a directed graph and $(A,B)$ is a pair of disjoint subsets of $V(G)$, say that $(A,B)$ is \emph{well-directed} if either every edge of $G$ between $A$ and $B$ goes from $A$ to $B$, or every edge between $A$ and $B$ goes from $B$ to $A$. We show that if $G$ is an algebraic digraph defined by a single polynomial and $A$ and $B$ are two sets of vertices such that $[G]$ has few edges between $A$ and $B$, then  we can find large subset $A'\subset A$ and $B'\subset B$ such that $(A',B')$ is well-directed. This is the consequence of the fact that $G$ contains no member of $\mathcal{M}(2,s)$ for some constant~$s$. For simplicity, write $\mathcal{M}^{-}(s)$ and $\mathcal{M}^{+}(s)$ instead of $\mathcal{M}(2,s,1)$ and $\mathcal{M}(2,s,2)$, respectively. Note that each member of $\mathcal{M}^{-}(s)$ (and $\mathcal{M}^{+}(s)$) is composed of $s$  directed edges $(u_1,v_1),\dots,(u_s,v_s)$ such that $(u_{i},v_j)$ is not an edge for $i>j$ ($j>i$, respectively).

\begin{lemma}\label{lemma:sparsesubset}
Let $n,d$ be positive integers, then there exists $c_{1}=c_{1}(n,d)>0$ and $c_2=c_2(n,d)>0$ such that the following holds. Let $G$ be an algebraic digraph of complexity $(n,d,1)$ on $N$ vertices such that $d([G])\leq 1-\alpha$. Then $V(G)$ contains two disjoint subsets $A$ and $B$ such that $|A|\geq c_1\alpha N$, $|B|\geq c_2 \alpha^{n+1}N$, and $(A,B)$ is well-directed. 
\end{lemma}

\begin{proof}
Let $s=\binom{n+2d}{n}+1$ and $\beta=\frac{1}{6\cdot 3^{2s}}$. Note that $[G]$ is a strongly-algebraic graph of complexity $(n,2d)$. Indeed, if $G$ is defined by the polynomial $f:(\mathbb{F}^{n})^{2}\rightarrow \mathbb{F}$ of degree at most $d$, then the polynomial $f'(\mathbf{x},\mathbf{y})=f(\mathbf{x},\mathbf{y})\cdot f(\mathbf{y},\mathbf{x})$ defines $[G]$ and has degree at most $2d$. Therefore, by Lemma \ref{lemma:weakvc}, $[G]$ has weak-VC-dimension $(c,n)$ for some $c=c(n,d)$. Hence, by Lemma \ref{lemma:sparsesubgraph}, there exist two subsets $A_0$ and $B_0$ of $V(G)$ such that $|A_0|\geq \frac{\alpha}{4}N$, $|B_0|\geq c_0\alpha (\alpha\beta)^{n}N=c_2' \alpha^{n+1}N$ (where $c_2'=c_2'(n,d)$ depends only on $n$ and $d$), and in $[G]$, every vertex in $B_0$ has at most $\beta |A_0|$ neighbours in  $A_0$. We show that $c_1=\frac{1}{4\cdot 3^{2s}}$ and $c_2=c_2'/2$ suffices.

We define the sequence of sets $A_0\supset A_1\supset\dots$ and sequence of edges $f_1,f_2,\dots\in E(G)$ as follows. For $i=1,2,\dots$,
\begin{itemize}
    \item  $|A_i|\geq |A_0|3^{-i}$,
    \item  $f_i=(u_i,v_i)$ has one vertex in $A_{i-1}$ and one vertex in $B$
    \item  if $v_i\in B_0$, then $v_i$ has no in-neighbor in $A_{i}$, and if $u_i\in B_0$, then $u_i$ has no out-neighbor in $A_{i}$.
\end{itemize}
 Suppose that $A_{i-1}$ is already defined for some $1\leq i<2s$. Consider two cases. First, consider the case that for every $b\in B_0$, either $N^{+}(b)\cap A_{i-1}=\emptyset$ or $N^{-}(b)\cap A_{i-1}=\emptyset$. Then there exists $B\subset B_0$ such that $|B|\geq |B_0|/2$ and either every $b\in B$ satisfies $N^{+}(b)\cap A_{i-1}=\emptyset$, or every $b\in B$ satisfies $N^{-}(b)\cap A_{i-1}=\emptyset$. Set $A=A_{i-1}$, then $(A,B)$ is well-directed. As $|A_{i-1}|> |A_0|3^{-2s}\geq c_1 \alpha N$, and $|B|\geq c_2 N$, the sets $A$ and $B$ satisfy our conditions. In this case, we stop.
 
 Now consider the case when there exists $b\in B_0$ such that both $N^{+}(b)\cap A_{i-1}$ and $N^{-}(b)\cap A_{i-1}$ are nonempty. If $a\in N^{+}(b)\cap N^{-}(b)\cap A_{i-1}$, then $\{a,b\}\in E([G])$, so $|N^{+}(b)\cap N^{-}(b)\cap A_{i-1}|\leq \beta|A_0|<\frac{1}{6}|A_{i-1}|$. But then, either $|N^{+}(b)\cap A_{i-1}|\leq \frac{2}{3}|A_{i-1}|$, or $|N^{-}(b)\cap A_{i-1}|\leq \frac{2}{3}|A_{i-1}|$. In the first case, set $A_i=A_{i-1}\setminus N^{+}(b)$ and $f_{i}=(b,a)$ for some $a\in N^{+}(b)$, in the second case $A_i=A_{i-1}\setminus N^{-}(b)$ and $f_{i}=(a,b)$ for some $a\in N^{-}(b)$. Then $A_i$ and $f_i$ satisfy the desired properties.

Note that we must have stopped for some $i\leq 2s$. Indeed, either at least half of the edges $f_1,\dots,f_i$ goes from $A_0$ to $B_0$, or at least half of the edges goes from $B_0$ to $A_0$. In the first case, $G$ contains a member of $\mathcal{M}^{-}(\lceil i/2\rceil)$, in the second case $G$ contains a member of $\mathcal{M}^{+}(\lceil i/2\rceil)$. 
\end{proof}

The following theorem will immediately imply Theorem \ref{thm:multicolor}, which in turn implies Theorem \ref{thm:main}.

\begin{theorem}\label{thm:central}
There exists $c>0$ such that the following holds. Let $n,d,m,N$ be positive integers, $\mathbb{F}$ be a field and $0<\beta<1$. Let $f_1,\dots,f_m: (\mathbb{F}^{n})^{2}\rightarrow \mathbb{F}$ be polynomials of degree at most $d$, and let $V\subset \mathbb{F}^{n}$  such that $|V|=N$. For $I\subset [m]$, let $G_{I}$ be the digraph defined on  $V$ such that $(\mathbf{x},\mathbf{y})$ is an edge if $f_{i}(\mathbf{x},\mathbf{y})\neq 0$ for $i\in I$, and $f_{i}(\mathbf{x},\mathbf{y})= 0$ for $i\not\in I$. If $N$ is sufficiently large, then either $G_{\emptyset}$ contains a clique of size at least $N^{1-\beta}$, or there exists $I\subset [m]$, $I\neq \emptyset$ such that $G_{I}$ contains a clique of size at least $\frac{1}{4m}N^{\beta/\gamma}$, where
 $$\gamma=c n m\min\left\{d,\frac{n\log d}{\log n}\right\}.$$
\end{theorem}

\begin{proof}
For $i=1,\dots,m$, let $G_{i}$ be the algebraic digraph on $V$ defined by $f_{i}$. Then $G_{I}=(\bigcap_{i\in I}G_{i})\cap (\bigcap_{i\in [m]\setminus I}\overline{G}_{i})$. Here, $\overline{G}_i$ is the digraph in which every directed edge is included which is not present in~$G_i$.

Let $0<\alpha<1$, $s_{0,1}=\dots=s_{0,m}=\binom{n+d}{d}+1$, and $U_{0}=V$. In what comes, we define the nested sequence of sets $U_{0}\supset U_{1}\supset\dots$ and sequence of positive integers $s_{\ell,i}$ for $\ell=0,1,\dots$ and $i\in [m]$ such that $G_{i}[U_{\ell}]$ contains no member of $\mathcal{M}^{+}(s_{\ell,i})$. By Lemma \ref{lemma:matching}, this is satisfied for $\ell=0$. Suppose that $U_{\ell}$ and $s_{\ell,i}$ has been already defined for some $\ell\geq 0$. 

Let $I\subset [m]$ be the set of indices $i$ such that $s_{\ell,i}>1$. Consider three cases.
\begin{description}
    \item[Case 1.] $I=\emptyset$. In this case we stop and set $U=U_{\ell}$. Note that $G_{i}[U]$ is empty for $i=1,\dots,m$, so $G_{\emptyset}[U]$ is a clique.
    
    \item[Case 2.] For every $i\in I$, we have $d([G_{i}][U_\ell])\geq 1-\alpha$. Let $G^{*}_{I}=\bigcap_{i\in I} [G_{i}]$, then $d(G^{*}_{I}[U_{\ell}])\geq 1-\alpha|I|\geq 1-\alpha m$. Therefore, by Lemma \ref{lemma:clique}, $G^{*}_{I}[U_{\ell}]$ contains a clique $U$ of size at least $\frac{1}{4\alpha m}$. Note that $G_{i}[U]$ is a clique for every $i\in I$, and $G_{i}[U]$ is an independent set for every $i\in [m]\setminus I$, so $G_{I}[U]$ is a clique. In this case, we stop as well.
    
    \item[Case 3.] There exists $i\in I$ such that $d([G_{i}][U_{\ell}])< 1-\alpha$. Then by Lemma \ref{lemma:sparsesubset} there exist two disjoint sets $A$ and $B$ in $U_{\ell}$ such that $|A|\geq c_1\alpha|U_{\ell}|$, $|B|\geq c_2 \alpha^{n+1}|U_{\ell}|$, and  $(A,B)$ is well-directed in $G_i$. Let $t=\lfloor s_{\ell,i}/n^{1/2}\rfloor$ and $u=s_{\ell,i}-t$. Note that if $G_{i}[A]$ contains a member of $\mathcal{M}^{+}(u)$, and $G_{i}[B]$ contains a member of $\mathcal{M}^{+}(t)$, then $G_{i}[U_{\ell}]$ contains a member of $\mathcal{M}^{+}(s_{l,i})$. Indeed, suppose that the directed edges $(x_1,y_1),\dots,(x_u,y_u)$ form a copy of a member of $\mathcal{M}^{+}(u)$ in $G_{i}[A]$, and $(x'_1,y'_1),\dots,(x'_t,y'_t)$ form a copy of a member of $\mathcal{M}^{+}(t)$ in $G_{i}[B]$.  If every edge between $A$ and $B$ goes from $A$ to $B$, then setting $x'_{i+t}=x_i$ and $y'_{i+t}=y_{i}$ for $i\in [u]$, the edges $(x'_1,y'_1),\dots,(x'_{u+t},y'_{u+t})$ form a copy of a member of $\mathcal{M}^{+}(s_{l,i})$. On the other hand, if every edge between $A$ and $B$ goes from $B$ to $A$, then setting $x_{i+u}=x'_i$ and $y_{i+u}=y'_{i}$ for $i\in [t]$, the edges $(x_1,y_1),\dots,(x_{u+t},y_{u+t})$ form a copy of a member of $\mathcal{M}^{+}(s_{l,i})$. 
    
    Therefore, either $G_{i}[A]$ contains no member of $\mathcal{M}^{+}(u)$, in which case set $U_{\ell+1}=A$, $s_{\ell+1,i}=u$, and say that $\ell$ is a \emph{small-jump}, or $G_{i}[B]$ contains no member of $\mathcal{M}(t)$, in which case set $U_{\ell+1}=B$, $s_{\ell+1,i}=t$, and say that $\ell$ is a \emph{big-jump}. For $j\in [m]\setminus \{i\}$, set $s_{\ell+1,j}=s_{\ell,j}$.
\end{description}
Let $s_{\ell}=s_{\ell,1}\dots s_{\ell,m}$. Let $L$ be the index $\ell$ at which we stop, let $a$ be the number of big-jumps, and $b$ be the number of small-jumps. Note that if $\ell$ is a small-jump, then $s_{\ell+1}\leq s_{\ell}(1-n^{-1/2}/2)$ and $|U_{\ell+1}|\geq c_1\alpha|U_{\ell}|$, and if  $\ell$ is a big-jump, then $s_{\ell+1}\leq s_{\ell}n^{-1/2}$ and $|U_{\ell+1}|\geq c_2\alpha^{n+1}|U_{\ell}|$. Therefore, $$s_{L}\leq s_{0}n^{-a/2}(1-n^{-1/2}/2)^{b}.$$ 
As $s_{L}\geq 1$ and $s_{0}=(\binom{n+d}{d}+1)^{m}\leq \min\{n^{dm},d^{nm}\}$, we get $$a\leq \frac{2\log s_0}{\log n}\leq 2m\min\left\{d,\frac{n\log d}{\log n}\right\}$$ and $$b\leq 4n^{1/2}\log s_0\leq 4n^{1/2}m\min\{d\log n,n\log d\}.$$
Also, we get 
$$|U_{L}|\geq N(c_1\alpha )^{b}(c_2\alpha^{n+1})^{a}>c_3\alpha^{(n+1)a+b}N,$$
where $c_3=c_3(n,m,d)>0$. Choosing $\alpha=N^{-\beta/\gamma'}$, where $\gamma'=16mn\min\{d,n\log d/\log n\}$, the left hand side is at least $N^{1-\beta}$ for sufficiently large $N$. Therefore, if we stopped in Case 1., then we found a set $U$ such that $G_{\emptyset}[U]$ is a clique of size $|U_{L}|\geq N^{1-\beta}$. On the other hand, if we stopped at Case 2., then we found  $I\subset [m]$, $I\neq [m]$ and a set $U$ such that $G_{I}[U]$ is a clique of size at least $\frac{1}{4m\alpha}\geq\frac{1}{4m}N^{\beta/\gamma'}$. This finishes the proof.
\end{proof}

\begin{proof}[Proof of Theorem \ref{thm:multicolor}]
 Consider an algebraic coloring of $K_{N}$ of complexity $(n,d,m)$ with $t$ colors. Let the polynomials $f_1,\dots,f_{m}$ define the coloring of $K_{N}$. Apply Theorem \ref{thm:central} with $\beta=1/2$, then there exists $I\subset [m]$ such that $G_{I}$ contains a clique of size at least $\frac{1}{4m}N^{1/2\gamma}$, where  $$\gamma=c n m\min\left\{d,\frac{n\log d}{\log n}\right\}.$$
 
 Note that for every $I\subset [m]$, the graph $G_{I}$ is monochromatic, hence, $K_{N}$ contains a monochromatic clique of size at least $\frac{1}{4m}N^{1/2\gamma}$.
\end{proof}

\section{Regularity lemma for algebraic hypergraphs}\label{sect:regularity}

In this section, we prove Theorem \ref{thm:regularity}. We remark that in the rest of the paper, we shall only work with undirected hypergraphs. We prepare the proof of Theorem \ref{thm:regularity} with several lemmas. The following regularity lemma is proved in \cite{FPS19} for hypergraphs of bounded VC-dimension, however, the same proof works almost word-by-word for hypergraphs of bounded weak-VC-dimension as well.

\begin{lemma}\label{lemma:ultrastrong}
Let $r$ be a positive integer and $c,n>0$, then there exists $c'=c'(r,c,n)>0$ such that the following holds. Let $0<\epsilon<1/4$. Let $\mathcal{H}$ be an $r$-uniform hypergraph of weak-VC-dimension $(c,n)$. Then $V(\mathcal{H})$ has an equitable partition $V_1,\dots,V_K$ with $8/\epsilon<K<c'(1/\epsilon)^{2n+1}$ parts such that all but at most $\epsilon$-fraction of the $r$-tuples of parts are $\epsilon$-homogeneous.
\end{lemma}

Let us restate the definition of $\mathcal{M}(r,s)$ for undirected hypergraphs. The hypergraph $\mathcal{M}$ is a member of $\mathcal{M}(r,s)$ if the vertex set of $\mathcal{M}$ is the union of  (not necessarily disjoint) $(r-1)$-element sets $Z_{1},\dots,Z_{s}$ and $s$ vertices $z_{1},\dots,z_{s}$, $Z_{i}\cup \{z_{i}\}$ is an edge of $\mathcal{M}$ for $i\in [s]$, and $Z_{i}\cup \{z_{j}\}$ is not an edge for $1\leq i<j\leq s$. The rest of the $r$-tuples can be either edges or non-edges. If $Z_1,\dots,Z_{s}$ and $z_1,\dots,z_s$ satisfy these properties, say that \emph{they define a member of $\mathcal{M}(r,s)$}.

In what follows, we show that if in addition we also assume that $\mathcal{H}$ contains no member of $\mathcal{M}(r,s)$, then those $r$-tuples of parts that have density at most $\epsilon$ can be made empty by removing a few vertices. We first need the following technical lemma, which can be viewed as a variant of Claim \ref{claim:med_degree} for sparse hypergraphs.

\begin{lemma}\label{lemma:med_degree2}
Let $\mathcal{G}$ be a nonempty $r$-partite $r$-uniform hypergraph with vertex classes $W_1,\dots,W_r$ such that $d(W_1,\dots,W_r)\leq \epsilon$. Then there exist $\ell\in [r]$ and $X\in\bigcup_{i\in [r]\setminus\{\ell\}}W_{i}$ such that $|X|=r-1$ and $1\leq |N(X)|\leq \epsilon^{1/r}|W_{\ell}|$.
\end{lemma}

\begin{proof}
We proceed by induction on $r$. In case $r=1$, the statement is trivial, so assume $r\geq 2$.

Let $U\subset V(\mathcal{G})$ be the set of vertices with positive degree, and let $U_i=W_i\cap U$ for $i\in [r]$. Then $|U_1|\cdots|U_r|\geq \epsilon |W_1|\cdots|W_{r}|$, so there exists $j\in [r]$ such that $|U_j|\geq \epsilon^{1/r}|W_j|$. Without loss of generality, assume that $j=r$. Then by simple averaging, there exists $x\in U_r$ such that 
 $$|N(x)|\leq \frac{|E(\mathcal{G})|}{|U_r|}\leq \epsilon^{(r-1)/r}|W_1|\dots |W_{r-1}|.$$
 
 Let $\mathcal{G}'$ be the $(r-1)$-partite $(r-1)$-uniform hypergraph with vertex classes $W_1,\dots,W_{r-1}$ formed by the neighborhood of $x$. By our induction hypothesis, there exists $\ell\in [r-1]$ and $X'\in\bigcup_{i\in [r-1]\setminus\{\ell\}}W_{i}$ such that $|X'|=r-2$ and $1\leq |N(X')|\leq \epsilon^{1/r}|W_{\ell}|$. But then $\ell$ and $X=X'\cup\{x\}$ satisfy our desired properties.
\end{proof}

The next lemma will be the heart of the proof of Theorem \ref{thm:regularity}. Before we state it, let us introduce a definition. Let $\mathcal{H}$ be an $r$-uniform hypergraph and let $(V_{1},\dots,V_{K})$ be disjoint subsets of its vertex set. Suppose that $Z_{1},\dots,Z_{s}\in V(\mathcal{H})^{(r-1)}$ and $z_{1},\dots,z_{s}\in V(\mathcal{H})$ define a member of $\mathcal{M}(r,s)$ in $\mathcal{H}$. If $z_{1},\dots,z_{s}$ belong to the same part $V_{j}$, and no two vertices of $Z_{i}\cup\{z_{i}\}$ belong to the same part for $i\in [s]$, then say that $\bigcup_{i=1}^{s} Z_{i}\cup\{z_1,\dots,z_{s}\}$ induces a \emph{focused} copy of a member of $\mathcal{M}(r,s)$, and $Z_{1},\dots,Z_{s}$ and $z_{1},\dots,z_{s}$ defines a \emph{focused} member of $\mathcal{M}(r,s)$.

\begin{lemma}\label{lemma:cleaning}
 Let $r,s$ be positive integers, then there exists $c=c(r,s)>0$ such that the following holds. Let $\epsilon_0>0$, let $\mathcal{H}$ be an $r$-uniform hypergraph, and let $U_{1},\dots,U_{L}$ be disjoint subsets of $V(\mathcal{H})$. Suppose that $\mathcal{H}$ contains no focused copy of a member of $\mathcal{M}(r,s)$. Let $\mathcal{G}$ be an $r$-uniform hypergraph on $[L]$ such that if $\{i_1,\dots,i_r\}$ is an edge of $\mathcal{G}$, then $d_{\mathcal{H}}(U_{i_1},\dots,U_{i_r})\leq \epsilon_0$. Then for $i\in [L]$, there exists $V_{i}\subset U_{i}$ such that $|V_{i}|\geq (1-c\epsilon_0^{1/r!})|U_i|$, and $(U_{i_1},\dots,U_{i_r})$ is empty for every $\{i_1,\dots,i_r\}\in E(\mathcal{G})$.
\end{lemma}

\begin{proof}
We prove this by induction on $r$. We proceed with the base case $r=2$ and the general case $r\geq 3$ simultaneously. Let $T:=\emptyset$, and let $\epsilon_1=\epsilon_0^{1/r}$. Let $X\in [L]^{(r-1)}$ and $k\in [L]\setminus X$. If $X\cup \{k\}\not\in E(\mathcal{G})$, set $F_{X,k}=\emptyset$. Otherwise, let $F_{X,k}$ be the $(r-1)$-partite $(r-1)$-uniform hypergraph on $\bigcup_{x\in X}U_{x}$, whose edges are the $(r-1)$-tuples $Z\in V(\mathcal{H})^{(r-1)}$ such that $Z$ has one vertex in $U_{x}$ for $x\in X$, and $|N_{\mathcal{H}}(Z)\cap U_{k}|\geq \epsilon_{1} |U_{k}|$. As $d_{\mathcal{H}}(\{U_{x}\}_{x\in X},U_{k})\leq \epsilon_{0}$, we have 
\begin{equation}
\label{eqn1}    
d_{F_{X,k}}(\{U_{x}\}_{x\in X})\leq \frac{\epsilon_{0}}{\epsilon_{1}}=\epsilon_1^{r-1}.
\end{equation}
Also, define $T_{k}\subset U_{k}$ as the union of the sets $U_{k}\cap N_{\mathcal{H}}(Z')$, where $Z'\in (V(\mathcal{H})\setminus U_{k})^{(r-1)}$, has all its vertices in different parts 
and $|U_{k}\cap N_{\mathcal{H}}(Z')|\leq \epsilon_{1}|U_{k}|$. 
\begin{claim}
$$|T_{k}|\leq s\epsilon_{1}|U_{k}|.$$ 
\end{claim}
\begin{proof}
Suppose that $|T_{k}|>s\epsilon_{1}|U_{k}|$. Then we can find greedily $Z_{1},\dots,Z_{s}\in (V(\mathcal{H})\setminus U_{k})^{(r-1)}$ and $z_{1},\dots,z_{s}\in T_{k}$ such that $z_{i}\in N(Z_{i})$ for $i\in [s]$ and $z_{i}\not\in N(Z_{j})$ for $i>j$. But then $Z_{1},\dots,Z_{s}$ and $z_1\dots,z_s$ induces a focused copy of a member of $\mathcal{M}(r,s)$, contradiction.
\end{proof}

For $k\in [L]$, let $U'_{k}=U_{k}\setminus T_{k}$. Note that if $X\cup \{k\}\in E(\mathcal{G})$ and $Z\cup \{z\}$ is an edge of $\mathcal{H}$ for some $Z$ having one vertex in every $U'_{x}, x\in X$  and $z\in U'_{k}$, then $Z\in E(F_{X,k})$. Remove every edge $Z$ from $F_{X,k}$ which has no neighbor in $U'_{k}$. Let $F_{X}=\bigcup_{k\in [L]\setminus X}F_{X,k}$.

Now let us consider the base case $r=2$. Then every $F_{X,k}$ is a just a subset of $U_X$, which by $(\ref{eqn1})$ has size at most $\epsilon_1|U_X|$. This implies
that $|F_{X}|\leq s\epsilon_1 |U_{X}|$ for $X\in [L]$. Indeed, otherwise, by a simple greedy argument, we can find $z_1,\dots,z_{s}\in F_{X}$ and $k_1,\dots,k_s$ such that $z_{i}\in F_{X,k_i}$ and $z_{i}\not\in F_{X,k_j}$ for $i>j$. Choose an arbitrary vertex $w_{i}\in N(z_{i})\cap U'_{k_i}$, then $w_1,\dots,w_{s}$ and $z_1,\dots,z_n$ define a focused copy of a member of $\mathcal{M}(2,s)$, contradiction. For $i\in [L]$, set $V_{i}=U'_{i}\setminus F_{i}$. Then $|V_{i}|\geq |U_{i}|(1-2s\epsilon_1)$, and there are no edges between $V_{i}$ and $V_j$ for $\{i,j\}\in E(\mathcal{G})$. Therefore, choosing $c(2,s)=2s$ proves this case.
Now consider the general case $r\geq 3$.

\begin{claim}
$$d_{F_{X}}(\{U_{x}\}_{x\in X})<3rs\epsilon_1.$$
\end{claim}
\begin{proof}
Let $F:=F_{X}$, let $W_{x}:=U'_{x}$ for $x\in X$. In what follows, step-by-step, we define $|X|$ sequences of pairs $(Z_{i,x},z_{i,x})\in (\bigcup_{y\in X\setminus\{x\}} U'_{y})^{(r-1)}\times U'_{x}$ for $x\in X$ and $i=0,1,\dots$, while removing elements from $W_{x}$ and edges from $F$. Suppose that $j_x$ is the largest index $i$ for which $(Z_{i,x},z_{i,x})$ is already defined.  If $(\{W_{x}\}_{x\in X})$ spans no edge in $F$, then stop. Otherwise, there exists $k\in [L]$ such that $F_{X,k}$ is not empty on $(\{W_{x}\}_{x\in X})$. By applying Lemma \ref{lemma:med_degree2} to the $(r-1)$-tuple $(\{U_x\}_{x\in X})$ and hypergraph $F_{X,k}\cap F$ (which by $(\ref{eqn1})$ has density at most $\epsilon_1^{r-1}$), there exist $x\in X$ and $Z\in \bigcup_{y\in X\setminus\{x\}} W_{y}$ such that $1\leq |N_{F_{X,k}\cap F}(Z)\cap W_{x}|\leq \epsilon_1|U_{x}|$. Let $z_{j_{x}+1,x}$ be an arbitrary element of $N_{F_{X,k}\cap F}(Z)\cap W_{x}$. Furthermore, remove the elements of $N_{F_{X,k}\cap F}(Z)$ from $W_{x}$, and remove the edges of $F_{X,k}$ from $F$. Also, let 
$Z_{j_{x}+1,x}:=Z$ and define $k_{j_x+1,x}:=k$.

After we stop, let $s_x$ be the largest index $i$ such that $(Z_{i,x},z_{i,x})$ is defined. Note that for every $x\in X$ and $i\leq s_x$, $Z_{1,x},\dots,Z_{i,x}$ and $z_{1,x},\dots,z_{i,x}$ define a focused copy of a member of $\mathcal{M}(r-1,i)$. Let $u_{i,x}$ be an arbitrary element of $U'_{k_{i,x}}\cap N_{\mathcal{H}}(Z_{i,x}\cup\{z_{i,x}\})$. Then $Z_{1,x}\cup\{u_{1,x}\},\dots,Z_{i,x}\cup\{u_{i,x}\}$ and $z_{1,x},\dots,z_{i,x}$ define a focused copy of a member of $\mathcal{M}(r,i)$ in $\mathcal{H}$. Therefore, we have $s_x<s$. Also, at each step when we modify $W_x$, we remove at most $\epsilon_1|U_{x}|$ elements, so after we stop, we have $|W_x|\geq |U'_{x}|-s_{x}\epsilon_1|U_{x}|$.

Let us bound $d_{F_{X}}(\{U_{x}\}_{x\in X})$. Recall that $d_{F_{X,k_{i,x}}}(\{U_{x}\}_{x\in X}) \leq \epsilon_1^{r-1}$ by $(\ref{eqn1})$. Let $F^{*}=\bigcup_{x\in X}\bigcup_{i\in [s_x]} F_{X,k_{i,x}}$. Then 
$$d_{F^{*}}(\{U_{x}\}_{x\in X})<\sum_{x\in X}\sum_{i\in [s_x]} d_{F_{X,k_{i,x}}}(\{U_{x}\}_{x\in X})<s(r-1)\epsilon_1^{r-1}.$$ 
If $k\neq k_{i,x}$ for every $x\in X$ and $i\in [s_x]$, then every edge of $F_{k,X}$ touches $U_{x}\setminus W_{x}$ for some $x\in X$.
Note that $|U_{x}\setminus W_{x}|=|U_{x}\setminus U'_{x}|+|U'_{x}\setminus W_{x}| \leq 2s\epsilon_1|U_x|$. Hence 
$\frac{|U_{x}\setminus W_{x}|}{|U_{x}|}< 2s\epsilon_1$. Therefore, we have
$$d_{F_{X}\setminus F^{*}}(\{U_{x}\}_{x\in X})<2(r-1)s\epsilon_1.$$
Finally, by our construction, $(\{W_{x}\}_{x\in X})$ spans no edges.  Thus, $d_{F_{X}}(\{U_{x}\}_{x\in X})<3rs\epsilon_1.$
\end{proof}
 
\noindent 
Since $|U'_x| \geq (1-s\epsilon_1)|U_x|$, we can write $\prod_{x \in X} |U'_x| \geq (1-s\epsilon_1)^{r-1}\prod_{x \in X} |U_x| \geq \frac{1}{2} \prod_{x \in X} |U_x|$,
assuming  $s\epsilon_1<1/10r$. Therefore, from the last claim we have that $d_{F_{X}}(\{U'_{x}\}_{x\in X})\leq 6rs\epsilon_1$.

Let $\mathcal{H}'=\bigcup_{X\in [L]^{(r-1)}}F_{X}$. We show that $\mathcal{H}'$ contains no focused copy of a member of $\mathcal{M}(r-1,s)$ with respect to the disjoint sets $U'_1,\dots,U'_{L}$. Otherwise, there exists $W_{1},\dots,W_{s}\in V(\mathcal{H}')^{(r-2)}$ and $w_{1},\dots,w_{s}\in U'_j$ for some $j\in [L]$ defining a focused member of $\mathcal{M}(r-1,s)$. Let $Z_{i}=W_{i}\cup\{w_i\}$. But then for $i\in [s]$, there exists $k_i\in [L]$ and $X_{i}\in [L]^{(r-1)}$ such that $Z_{i}\in F_{X_{i},k_i}$. Picking $z_{i}\in N_{\mathcal{H}}(Z_{i})\cap U_{i}'$ arbitrarily, $\{W_{i}\cup\{z_i\}\}_{i\in [s]}$ and $\{w_{i}\}_{i\in [s]}$ define a focused member of $\mathcal{M}(r,s)$ in $\mathcal{H}$, contradiction.
    
But then we can apply our induction hypothesis with the following parameters: the $(r-1)$-uniform hypergraph $\mathcal{H}'$ instead of $\mathcal{H}$, the disjoint sets $U_1',\dots,U_{L}'$ instead of $U_1,\dots,U_L$, $6rs\epsilon_1$ instead of $\epsilon_0$, and the complete $(r-1)$-uniform hypergraph on $[L]$ instead of $\mathcal{G}$. 
Recalling that $|U'_i| \geq (1-s\epsilon_1)|U_i|$ and $\epsilon_1=\epsilon_0^{1/r}$, we conclude that for $i\in [L]$, there exists $V_{i}\subset U_{i}'$ such that $$|V_{i}|\geq (1-c(r-1,s)(6rs\epsilon_1)^{1/(r-1)!})|U_{i}'|\geq (1-c(r,s)\epsilon_0^{1/r!})|U_{i}|$$ for some suitable $c(r,s)>0$, and $(V_{i_1},V_{i_2},\dots,V_{i_{r-1}})$ is empty for $(i_1,\dots,i_{r-1})\in [L]^{(r-1)}$. But note that then $(V_{i_1},\dots,V_{i_r})$ is also empty in $\mathcal{H}$ if $\{i_1,\dots,i_r\}\in E(\mathcal{G})$. This finishes the proof.
\end{proof}

Now everything is set to prove the main theorem of this section.

\begin{proof}[Proof of Theorem \ref{thm:regularity}]
 Let $\epsilon_0>0$, which we shall specify later as a function of $\epsilon$. By Lemma \ref{lemma:weakvc} and Lemma \ref{lemma:ultrastrong}, there exists an equitable partition of  $V(\mathcal{H})$ into $L$ parts $U_1,\dots,U_L$  such that $8/\epsilon_{0}<L<c'(1/\epsilon_{0})^{2n+1}$ and all but at most $\epsilon_{0}$-fraction of the $r$-tuples of parts are $\epsilon_{0}$-homogeneous.
 
Let $s=\binom{n+d}{d}+1$, then $\mathcal{H}$ contains no member of $\mathcal{M}(r,s)$ by Lemma \ref{lemma:matching}. In particular, $\mathcal{H}$ contains no focused copy of a member of $\mathcal{M}(r,s)$ with respect to the partition $U_1,\dots,U_L$. Let $c_0=c(r,s)$ be the constant given by Lemma \ref{lemma:cleaning}. Then for every $i\in [L]$, there exists $V'_{i}\subset U_{i}$ such that $$|V'_{i}|\geq (1-c_0 \epsilon_0^{1/r!})|U_{i}|\geq (1-c_0 \epsilon_0^{1/r!})\left\lfloor \frac{N}{L}\right\rfloor,$$ and $(V'_{i_1},\dots,V'_{i_r})$ is empty if $d(U_{i_1},\dots,U_{i_r})<\epsilon_0$. Let $K$ be a minimal positive integer such that $N/K<\min_{i\in [L]}|V_{i}'|$, then $K<L+2c_{0}\epsilon_{0}^{1/r!}L<2L$. 
In particular, when we choose $\epsilon_0$ it will satisfy $c_{0}\epsilon_{0}^{1/r!}<1/2$. Let $V_1,\dots,V_{K}$ be any equitable partition of $V(\mathcal{H})$ such that $V_{i}\subset V_{i}'$ for $i\in [L]$, this clearly exists by the definition of $K$.

Say that an $r$-tuple $\{i_1,\dots,i_r\}\in [K]^{(r)}$ is \emph{good} if $d(V_{i_1},\dots,V_{i_r})>1-2^{r}\epsilon_0$, or $(V_{i_1},\dots,V_{i_r})$ is empty. Otherwise, say that $\{i_1,\dots,i_r\}$ is \emph{bad}. Note that if $\{i_1,\dots,i_r\}\in [L]^{r}$, then $\{i_1,\dots,i_r\}$ is good if $(U_{i_1},\dots,U_{i_r})$ is $\epsilon$-homogeneous. Indeed, if $d(U_{i_1},\dots,U_{i_r})>1-\epsilon_0$, then $$d(V_{i_1},\dots,V_{i_r})>1-\frac{|U_{i_1}|\dots |U_{i_r}|}{|V_{i_1}|\dots|V_{i_r}|}\epsilon_0\geq 1-2^{r}\epsilon_0,$$ and if $d(U_{i_1},\dots,U_{i_r})<\epsilon_0$, then $(V_{i_1},\dots,V_{i_r})$ is empty by definition. The number of $r$-tuples $\{i_1,\dots,i_r\}\in [K]^{(r)}$ having a nonempty intersection with $\{L+1,\dots,K\}$ is less than $(K-L)K^{r-1}$. Hence, the total number of bad $r$-tuples is at most 
$$\epsilon_0\binom{L}{r}+(K-L)K^{r-1}< \epsilon_0 \binom{K}{r}+2c_0\epsilon_{0}^{1/r!} K^{r}\leq c_{1}\epsilon_{0}^{1/r!}\binom{K}{r},$$ 
where $c_{1}=c_1(r,s)>c_0$ depends only on $r$ and $s$. Choose $\epsilon_0$ such that both $c_1\epsilon_0^{1/r!},  2^{r}\epsilon_0 \leq \epsilon<1/4$. Then $V_1,\dots,V_K$ is an equitable partition of $\mathcal{H}$ such that $8/\epsilon<8/\epsilon_0<K\leq 2c'(1/\epsilon_0)^{2n+1}<c(1/\epsilon)^{r!(2n+1)}$ for some $c=c(r,s)>0$, and all but at most an $\epsilon$-fraction of the $r$-tuples of parts $(V_{i_1},\dots,V_{i_r})$ satisfy that either $d(V_{i_1},\dots,V_{i_r})>1-2^{r}\epsilon_0>1-\epsilon$, or $(V_{i_1},\dots,V_{i_r})$ is empty.
\end{proof}

\section{A Ramsey-type result for algebraic hypergraphs}\label{sect:hypergraph}

In this section, we prove Theorem \ref{thm:hereditary}. First, let us extend the list of forbidden hypergraphs in algebraic hypergraphs of constant complexity. In order to do this, we use a result about the maximum flattening rank of tensors which are almost diagonal. Let $T:A_1\times\dots\times A_{r}\rightarrow \mathbb{F}$ be an $r$-dimensional tensor. Let the \emph{max-flattening rank} of $T$, denoted by $\mfrank(T)$, be the maximum of $\frank_i(T)$ for $i\in [r]$. Say that an $r$-dimensional tensor $T:A^{r}\rightarrow \mathbb{F}$ is \emph{semi-diagonal} if the following holds. Let $a_1,\dots,a_r\in A$, then $T(a_1,\dots,a_{r})=0$ if  $a_1,\dots,a_r$ are pairwise distinct, and $T(a_1,\dots,a_r)\neq 0$ if $a_1=\dots=a_d$. If $a_1,\dots,a_r$ are neither all equal or all distinct, there is no restriction on $T(a_1,\dots,a_r)$. We prove the following result in the companion note \cite{frank}.

\begin{theorem}\label{thm:semidiag}
Let $T:A^{r}\rightarrow \mathbb{F}$ be an $r$-dimensional semi-diagonal tensor. Then $$\mfrank(T)\geq \frac{|A|}{r-1}.$$
\end{theorem}

Let $\mathcal{N}_{r,s}$ be the family of $r$-uniform hypergraphs $\mathcal{N}$ having the following form. There are $s$ disjoint edges $f_{i}=\{u_{i,1},\dots,u_{i,r}\}$ for $i\in [s]$, whose union is the vertex set of $\mathcal{N}$, and the $r$ element set $\{u_{i_1,1},\dots,u_{i_r,r}\}$ is not an edge for any $r$-tuple $(i_1,\dots,i_r)\in [s]^{r}$, where $i_1,\dots,i_r$ are pairwise distinct. The rest of the $r$ element subsets of the vertex set can be either edges or non-edges.

\begin{lemma}\label{lemma:forbidden2}
Let $\mathcal{H}$ be a strongly-algebraic $r$-uniform hypergraph of complexity $(n,d)$. Then $\mathcal{H}$ contains no member of $\mathcal{N}_{r,s}$ for $s>(r-1)\binom{n+d}{d}$.
\end{lemma}

\begin{proof}
Let $V=V(\mathcal{H})\subset \mathbb{F}^{n}$, and let $f:(\mathbb{F}^{n})^r\rightarrow \mathbb{F}$ be the polynomial defining $\mathcal{H}$. Define tensor $T:V^{r}\rightarrow \mathbb{F}$ to be $T(\mathbf{x}_1,\dots,\mathbf{x}_r)=f(\mathbf{x}_1,\dots,\mathbf{x}_r)$. Then by Lemma \ref{lemma:tensorrank} $\mfrank(T)\leq \binom{n+d}{d}$. 

Suppose that $\mathcal{H}$ contains a member of $\mathcal{N}_{r,s}$ for some $s>(r-1)\binom{n+d}{n}$, then there exists $\mathbf{u}_{i,j}\in \mathbb{F}^{n}$ for $(i,j)\in [s]\times [r]$ such that $f(\m{u}_{i,1},\dots,\m{u}_{i,r})\neq 0$ for $i\in [s]$, and $f(\m{u}_{i_1,1},\dots,\m{u}_{i_r,r})=0$ for $(i_1,\dots,i_r)\in [s]^{r}$ for which $i_1,\dots,i_r$ are pairwise distinct. Let $T'$ be the subtensor of $T$ induced on $\{\m{u}_{1,1},\dots,\m{u}_{s,1}\}\times\dots\times \{\m{u}_{1,r},\dots,\m{u}_{s,r}\}$. Identifying $\m{u}_{i,j}$ by $i$, the tensor $T':[s]^{r}\rightarrow \mathbb{F}$ is semi-diagonal, so $$\mfrank(T)\geq \mfrank(T')\geq \frac{s}{r-1}$$ by Theorem~\ref{thm:semidiag}. This is a contradiction, finishing the proof. 
\end{proof}

Now everything is set to prove the main theorem of this section.

\begin{proof}[Proof of Theorem \ref{thm:hereditary}]
Let $\mathcal{H}\in \mathcal{F}$ and let $N=|V(\mathcal{H})|$. Then $\mathcal{H}$ is an $r$-uniform strongly-algebraic hypergraph of complexity $(n,d)$. Let $\epsilon=N^{-\beta/r!(2n+1)}$, then by Theorem \ref{thm:regularity},  there is an equitable partition of $V(\mathcal{H})$ into parts $V_1,\dots,V_{K}$ for some $8/\epsilon<K<c'(1/\epsilon)^{r!(2n+1)}$ such that all but at most $\epsilon$-fraction of the $r$-tuples of parts are either empty, or have density at least $1-\epsilon$. Here, ${c'=c'(r,n,d)>0}$.

For $i\in [K]$, pick a vertex $v_i\in V_i$ randomly with uniform distribution. Let $\mathcal{H}'$ be the hypergraph induced on the vertex set $\{v_1,\dots,v_{K}\}$, then $\mathcal{H}'\in \mathcal{F}$ as well. Define the $r$-uniform hypergraph $\mathcal{G}$ on vertex set $[K]$ as follows. The  $r$-element set $\{i_1,\dots,i_r\}\in [K]^{(r)}$ is an edge of $\mathcal{G}$ if and only if either $\{v_{i_1},\dots,v_{i_r}\}\in E(\mathcal{H})$ and $d(V_{i_1},\dots,V_{i_r})\geq 1-\epsilon$, or $\{v_{i_1},\dots,v_{i_r}\}\not\in E(\mathcal{H})$ and $(V_{i_1},\dots,V_{i_r})$ is empty. Note that if $(V_{i_1},\dots,V_{i_r})$ is $\epsilon$-homogeneous, then $\{i_1,\dots,i_r\}$ is an edge with probability at least $1-\epsilon$. Therefore, $\mathbb{E}(d(\mathcal{G}))\geq 1-2\epsilon$, so there is a choice for $v_1,\dots,v_{K}$ such that $d(\mathcal{G})\geq 1-2\epsilon$. Fix such a choice. By Lemma \ref{lemma:clique}, $\mathcal{G}$ contains a clique $J$ of size at least $\frac{1}{8}\epsilon^{-1/(r-1)}$. Let $\mathcal{H}^{*}$ be the subhypergraph of $\mathcal{H}$ induced on the vertex set $\{v_j:j\in J\}$. As $\mathcal{H}^{*}\in\mathcal{F}$, $\mathcal{H}^{*}$ contains either an independent set of size $s$, or a clique of size $c|J|^{\alpha}=c_1N^{\alpha\beta/r!(r-1)(2n+1)}$ for some $c_1=c_1(r,n,d,c,\alpha,\beta)>0$. In the latter case, we are done, so assume that $\mathcal{H}^{*}$ contains an independent set $\{v_i:i\in I\}$ of size $s$.

 We finish the proof by noting that at least one of $U_{i}$, $i\in I$ is an independent set in $\mathcal{H}$. Otherwise, if $U_{i}$ contains an edge $f_i$ for every $i\in I$, then $\bigcup_{i\in I}f_i$ spans a member of $\mathcal{N}_{r,s}$ in $\mathcal{H}$, contradicting  Lemma \ref{lemma:forbidden2}. Therefore, $\mathcal{H}$ contains an independent set of size at least $\frac{1}{2}\lfloor N/K\rfloor>c_2N^{1-\beta}$, where $c_2=c_2(r,n,d,c,\alpha,\beta)>0$.
\end{proof}

\section{Concluding remarks}

Following Fox and Pach \cite{FP08}, we say that a family of graphs $\mathcal{G}$  has the \emph{strong-Erd\H{o}s-Hajnal property}, if there exists a constant $c=c(\mathcal{G})>0$ such that for every $G\in\mathcal{G}$, either $G$ or its complement contains a bi-clique of size at least $c|V(G)|$. In \cite{APPRS}, it is proved that the strong-Erd\H{o}s-Hajnal property implies the Erd\H{o}s-Hajnal property in hereditary graph families, and that the family of semi-algebraic graphs of complexity $(n,d,m)$ has the strong-Erd\H{o}s-Hajnal property. On the other hand, it is known that this is not the case for algebraic graphs. For infinitely many $N$, one describe a strongly-algebraic graph $G$ of complexity $(3,2)$ such that the size of the largest bi-clique in both $G$ and its complement is $O(N^{3/4})$.

If $q$ is a prime power, the Erd\H{o}s-R\'enyi graph $ER_{q}$, is defined as follows. The vertices of $ER_{q}$ are the elements of the projective plane over $\mathbb{F}_{q}$, and $(x_0,x_1,x_2)$ and $(y_0,y_1,y_2)$ are joined by an edge if $x_0y_0+x_1y_1+x_2y_2=0$. The graph $ER_{q}$ has $N=q^{2}+q+1$ vertices, is $(q+1)$-regular and has at most $2(q+1)$ vertices with loops. This graph contains no copy of $K_{2,2}$ and has eigenvalues $\pm\sqrt{q}$ and $q+1$, where the multiplicity of $q+1$ is $1$, see e.g. \cite{KS}. By the expander mixing lemma (see, e.g., Theorem 2.11 in \cite{KS}), one can bound the number of the edges/non-edges of $ER_{q}$ between two disjoint subsets of vertices using its eigenvalues. This lemma implies that both $ER_{q}$ and its complement contains no bi-clique of size larger than $$\frac{N\sqrt{q}}{q+1}=O(N^{3/4}).$$ Note that the complement of $ER_{q}$ is strongly-algebraic of complexity $(3,2)$. Indeed, we can view the vertices of $ER_{q}$ as elements of $\mathbb{F}_{q}^{3}$, by replacing each $(x_0,x_1,x_2)\in \mathbb{PF}_{q}^{2}$ with one element of the equivalence class $C_{(x_{0},x_{1},x_{2})}=\{(\lambda x_{0},\lambda x_{1},\lambda x_{2}):\lambda\in \mathbb{F}_{q}\setminus\{0\}\}\subset \mathbb{F}_{q}^3$. For a variant of this construction, see e.g. \cite{CS18+}, Section 6.1.

Finally, let us mention that although algebraic graphs of bounded complexity do not have strong Ramsey properties (as we proved in this paper), they are one of the main sources of best examples for so called Tur\'an-type questions. Given an $r$-uniform hypergraph $\mathcal{H}$, the \emph{extremal number (or Tur\'an number) of $\mathcal{H}$}, denoted by $\mbox{ex}(N,\mathcal{H})$ is the maximum number of edges in an $r$-uniform hypergraph on $N$ vertices which contains no copy of $\mathcal{H}$ as a subhypergraph. If $H$ is a graph, the asymptotic value of the extremal number of $H$ is known by the Erd\H{o}s-Stone theorem,  unless $H$ is bipartite. The case of bipartite graphs is notoriously hard. In many cases when the order of $\mbox{ex}(N,H)$ is known for some bipartite graph $H$, the construction achieving the right order of magnitude is an algebraic graph of bounded complexity, see e.g. \cite{B66,B15,BC18,ERS66}. Algebraic hypergraphs of bounded complexity also appear in connection to Tur\'an type results, see e.g. \cite{MYZ18, CPZ20}.

\vspace{0.3cm}
\noindent	
{\bf Acknowledgements.} We would like to thank Jacob Fox, J\'anos Pach, and Artem Chernikov for pointing out the related references and results. Both authors were supported by the SNSF grant 200021\_196965. Istv\'an Tomon also acknowledges the support of Russian Government in the framework of MegaGrant no 075-15-2019-1926, and the support of MIPT Moscow.

\end{document}